\documentclass[
final
]{dmtcs-episciences}

\usepackage[utf8]{inputenc}
\usepackage{subfigure}

\author{ Christoph Brause\affiliationmark{1,2}
  \and Michael A. Henning\affiliationmark{2}\thanks{Research supported in part by the South African National Research Foundation and the University of Johannesburg}
  \and Marcin Krzywkowski \affiliationmark{2,3}\thanks{Research fellow of the Claude Leon Foundation at the University of Johannesburg. Research partially supported by the Polish National Science Centre grant 2011/02/A/ST6/00201.}  
  }
\title{A characterization of trees with equal $2$-domination and $2$-independence numbers}
\affiliation{
  Institute of Discrete Mathematics and Algebra, TU Bergakademie Freiberg, Germany \\
  Department of Pure and Applied Mathematics, University of Johannesburg, South Africa\\
  Faculty of Electronics, Telecommunications and Informatics, Gdansk University of Technology, Poland
}
\received{2016-04-25}
\accepted{2017-1-12}

\usepackage{tikz}
\usetikzlibrary{shapes.geometric}
\tikzstyle{snode}=[circle,draw=black,fill=white,thick, inner sep=0pt ,minimum size=1.2mm]
\tikzstyle{gamma2}=[rectangle ,draw=black,inner sep=0pt ,minimum size=3.2mm]
\tikzstyle{bnode}=[circle ,draw=black,fill=black,thick, inner sep=0pt ,minimum size=1.2mm]
\tikzstyle{alpha2}=[diamond,draw=black,thick, inner sep=0pt ,minimum size=3.2mm]

\usepackage{geometry,graphicx,verbatim,amsmath}
\usetikzlibrary{calc}
\usepackage{cite}

\newcommand{\smallqed}{\hfill{\tiny ($\Box$)}}

\usepackage{amsthm}

\newtheorem{thm}{Theorem}

\newtheorem{obs}[thm]{Observation}
\newtheorem{lem}[thm]{Lemma}

\newtheorem{definition}[thm]{Definition}
\newtheorem{fact}{Fact}[thm]
\newtheorem{claim}[fact]{Claim}

\newenvironment{claimproof}{{\noindent \bfseries Proof:}}{\smallqed}

\newcommand{\gtd}{\gamma_2}
\newcommand{\tind}{\alpha_2}

\newcommand{\cT}{\mathcal{T}}
\newcommand{\cO}{{\cal O}}

\newcommand{\pdi}{{\rm{pdi}}}

\newcommand{\diam}{{\rm diam}}

\begin{document}
\publicationdetails{19}{2017}{1}{1}{1433}
\maketitle
\begin{abstract}
A set $S$ of vertices in a graph $G$ is a $2$-dominating set if every vertex of $G$ not in~$S$ is adjacent to at least two vertices in $S$, and $S$ is a $2$-independent set if every vertex in $S$ is adjacent to at most one vertex of $S$. The $2$-domination number $\gamma_2(G)$ is the minimum cardinality of a $2$-dominating set in $G$, and the $2$-independence number $\alpha_2(G)$ is the maximum cardinality of a $2$-independent set in $G$. Chellali and Meddah [{\it Trees with equal $2$-domination and $2$-independence numbers,} Discussiones Mathematicae Graph Theory 32 (2012), 263--270] provided a constructive characterization of trees with equal $2$-domination and $2$-independence numbers. Their characterization is in terms of global properties of a tree, and involves properties of minimum $2$-dominating and maximum $2$-independent sets in the tree at each stage of the construction. We provide a~constructive characterization that relies only on local properties of the tree at each stage of the construction.
\end{abstract}

\section{Introduction}

We continue the study of $2$-domination and $2$-independence in trees. For $k \ge 1$, a $k$-{\it dominating set} of a graph $G$ is a set $S$ of vertices of $G$ such that every vertex outside $S$ has at least $k$ neighbors in $S$, while $S$ is a $k$-{\it independent set} if every vertex in $S$ is adjacent to at most $k-1$ vertices of $S$. The {\it $k$-domination number} of $G$, denoted by $\gamma_k(G)$, is the minimum cardinality of a $k$-dominating set of $G$, and the {\it $k$-independence number} of $G$, denoted by $\alpha_k(G)$, is the maximum cardinality of a $k$-independent set of $G$. In particular, we note that for $k = 1$, a $1$-dominating set and a $1$-independent set are the classical dominating and independent sets, respectively. Thus, the $1$-domination number of $G$, $\gamma_1(G)$, is the domination number $\gamma(G)$ and the $1$-independence number of $G$, $\alpha_1(G)$, is the independence number $\alpha(G)$. A $k$-dominating set of $G$ of minimum cardinality is called a $\gamma_k(G)$-{\it set}, and a $k$-independent set of $G$ of maximum cardinality is called an $\alpha_k(G)$-{\it set}.

The concepts of $k$-domination and $k$-independence in graphs were introduced by Fink and Jacobson~\cite{fink} in 1985 and is now very well studied in the literature (see for example \cite{caro,chellali1,chellali2,delavina,desormeaux,favaron1,favaron2,favaron3,fujisawa,hansberg}). We refer the reader to the two books on domination by Haynes, Hedetniemi, and Slater \cite{funda1,funda2}, as well as to the excellent survey on $k$-domination and $k$-independence in graphs by Chellali, Favaron, Hansberg, and Volkmann \cite{chellali1}.

Fink and Jacobson \cite{fink} proved that $\gamma_2(G) \le \alpha_2(G)$ for every graph $G$. Recently, Chellali and Meddah \cite{chellali2} gave a constructive characterization of trees $T$ satisfying $\gamma_2(T) = \alpha_2(T)$. Their characterization is in terms of global properties of a tree, and involves properties of minimum $2$-dominating and maximum $2$-independent sets in the tree at each stage of the construction. We provide a constructive characterization that relies only on local properties of the tree at each stage of the construction.

\subsection{Notation}

For notation and graph theory terminology not defined herein, we refer the reader to \cite{HeYe_book}. Let $G$ be a graph with vertex set $V(G)$ of order $n(G) = |V(G)|$ and edge set $E(G)$ of size $m(G) = |E(G)|$. A {\it path} on $n$ vertices is denoted by $P_n$. For two vertices $u$ and $v$ in a connected graph $G$, the {\it distance} $d_G(u,v)$ between $u$ and $v$ is the length of a shortest $(u,v)$-path in $G$. The maximum distance among all pairs of vertices of $G$ is called the {\it diameter} of $G$, which is denoted by $\diam(G)$. A path of length $\diam(G)$ between two vertices at maximum distance apart in $G$ is a {\it diametrical path} of $G$. A vertex of degree one is called a {\it leaf} and its neighbor a {\it support vertex}. We denote the set of leaves of a tree $T$ by $L(T)$. A {\it star} is a tree $K_{1,k}$ for some $k \ge 1$, while for $r,s \ge 1$, a {\it double star} $S_{r,s}$ is a tree with exactly two vertices that are not leaves, one of which is adjacent to $r$ leaves and the other to $s$ leaves.

The {\it open neighborhood} of a vertex $v \in V(G)$ is the set $N_G(v) = \{u \in V(G) \colon uv \in E(G)\}$, and the {\it closed neighborhood of $v$} is $N_G[v] = N_G(v) \cup \{v\}$. The degree of $v$ is $d_G(v) = |N_G(v)|$. The {\it open neighborhood} of a set of vertices $S \subseteq V(G)$ is the set $N_G(S) = \bigcup_{v \in S} N_G(v)$, and the {\it closed neighborhood} of $S$ is $N_G[S] = N_G(S) \cup S$.

For a set $S \subseteq V(G)$, we let $G[S]$ denote the subgraph induced by $S$. The graph obtained from $G$ by removing the vertices of $S$ along with all edges incident to vertices in $S$ is denoted by $G-S$. If $S = \{v\}$, then we simply denote $G-S$ by $G-v$. We define the {\it boundary} of $S$, denoted by $\partial(S)$, to be the set of vertices of $S$ that have a neighbor in $V(G) \setminus S$.

A {\it rooted tree} $T$ distinguishes one vertex $r$ called the {\it root}. For each vertex $v \ne r$ of $T$, the {\it parent} of $v$ is the neighbor of $v$ on the unique $(r,v)$-path, while a {\it child} of $v$ is any other neighbor of $v$. The set of all children of $v$ we denote by $C(v)$. A {\it descendant} of $v$ is a vertex $u \ne v$ such that the unique $(r,u)$-path contains $v$. Thus every child of $v$ is a descendant of $v$. We let $D(v)$ denote the set of all descendants of $v$, and we define $D[v] = D(v) \cup \{v\}$. The {\it maximal subtree} at $v$ is the subtree of $T$ induced by $D[v]$, and it is denoted by $T_v^r$. If the root $r$ is clear from the context, then we simply denote the maximal subtree at $v$ by $T_v$.

\subsection{Known Results}

Fink and Jacobson \cite{fink} proved that $\gamma_2(G) \le \alpha_2(G)$ for every graph $G$, and conjectured that for every graph $G$ and integer $k \ge 1$ we have $\gamma_k(G) \le \alpha_k(G)$. Their conjecture was proven by Favaron \cite{favaron1} by the following stronger result.

\begin{thm}[\cite{favaron1}]\label{t:relate}
For every graph $G$ and integer $k \ge 1$, the graph $G$ contains a set that is both $k$-dominating and $k$-independent, and therefore $\gamma_k(G) \le \alpha_k(G)$.
\end{thm}

A graph $G$ that satisfies $\gamma_k(G) = \alpha_k(G)$ we call a $(\gamma_k,\alpha_k)$-{\it graph}. Recently, Chellali and Meddah \cite{chellali2} gave a constructive characterization of $(\gamma_2,\alpha_2)$-trees. For this purpose, they defined a family $\mathcal{O}$ of trees $T = T_i$ that can be obtained as follows. Let $\mathcal{O}$ be the family of trees that $T$ that can be obtained from a sequence $T_1,T_2,\ldots,T_k$ ($k \ge 1$) of trees, where $T_1$ is a star $K_{1,p}$ ($p \ge 1$), $T = T_k$, and, if $k \ge 2$, then $T_{i+1}$ is obtained recursively from $T_i$ by one of the following operations:
\begin{itemize}
\begin{item}
\textbf{Operation $\mathcal{R}_1$:} Add a star $K_{1,p}$, $p \ge 2$, centered at a vertex $u$ and join $u$ by an edge to a vertex of $T_i$.
\end{item}
\begin{item}
\textbf{Operation $\mathcal{R}_2$:} Add a double star $S_{1,p}$ with support vertices $u$ and $v$, where $|L_v| = p$ and join $v$ by an edge to a vertex $w$ of $T_i$ with the condition that if $\gamma_2(T_i-w) = \gamma_2(T_i)-1$, then no neighbor of $w$ in $T_i$ belongs to a $\gamma_2(T_i-w)$-set.
\end{item}
\begin{item}
\textbf{Operation $\mathcal{R}_3$:} Add a path $P_2 = u'u$ and join $u$ by an edge to a leaf $v$ of $T_i$ that belongs to every $\alpha_2(T_i)$-set and satisfies in addition $\alpha_2(T_i - v)+1 = \alpha_2(T_i)$.
\end{item}
\begin{item}
\textbf{Operation $\mathcal{R}_4$:} Add a path $P_3 = u'uv$ and join $v$ by an edge to a vertex $w$ that belongs to a $\gamma_2(T_i)$-set and satisfies further $\gamma_2(T_i-w) \le \gamma_2(T_i)$, with the condition that if $\gamma_2(T_i-w) = \gamma_2(T_i)-1$, then no neighbor of $w$ in $T_i$ belongs to a $\gamma_2(T_i-w)$-set.
\end{item}
\end{itemize}

We are now in a position to state the result due to Chellali and Meddah \cite{chellali2}.

\begin{thm}[\cite{chellali2}]\label{known}
A tree $T$ is a $(\gamma_2,\alpha_2)$-tree if and only if $T \cong K_1$ or $T \in \mathcal{O}$.
\end{thm}

\section{Main Result}

The Chellali and Meddah \cite{chellali2} characterization of $(\gamma_2,\alpha_2)$-trees presented in Theorem \ref{known} is a pleasing and important result. However, the characterization is not fully satisfactory in the sense that it is dependant on global properties of the tree at each stage of the construction. For example, in Operation $\mathcal{R}_2$ one needs to check that the tree $T_i$ and the vertex $w$ satisfy the condition that if $\gamma_2(T_i-w) = \gamma_2(T_i)-1$, then no neighbor of $w$ in $T_i$ belongs to a $\gamma_2(T_i-w)$-set. Operations $\mathcal{R}_3$ and $\mathcal{R}_4$ also require to check global properties involving minimum $2$-dominating and maximum $2$-independent sets in the tree. Motivated by the Chellali-Meddah construction of $(\gamma_2,\alpha_2)$-trees, our aim is to obtain a constructive characterization that relies only on local properties of the tree at each stage of the construction.We describe such a family $\cT$ of $(\gtd,\tind)$-trees in Section~\ref{family}. Our main result is the following constructive characterization of $(\gtd,\tind)$-trees. A proof of Theorem~\ref{t:tree} is presented in Section~\ref{S:mainp}.

\begin{thm}\label{t:tree}
A tree is $T$ a $(\gtd,\tind)$-tree if and only if $T \in \cT$.
\end{thm}

\section{The Family $\cT$}\label{family}

In this section, we define a family $\cT$ of $(\gtd,\tind)$-trees. For this purpose, we first define two sets of trees $A = \{T_1,\ldots,T_{15}\}$ and $B = \{B_1,\ldots,B_{10}\}$ shown in Figure~\ref{fig1A} and Figure~\ref{fig1B}, respectively. We call each tree that belongs to $A \cup B$ a \emph{special tree}.

\begin{figure}[h]
\begin{center}
\begin{minipage}{0.19\textwidth}
\centering
\begin{tikzpicture}
\draw(0,0.25)node[bnode]{}--(0.5,0)node[snode]{}--(0,-0.25)node[bnode]{};
\coordinate[label=right:$v$] (A) at (0.5,0);
\end{tikzpicture}
\end{minipage}
\begin{minipage}{0.19\textwidth}
\centering
\begin{tikzpicture}
\draw(-0.5,0)node[bnode]{}--(0,0)node[bnode]{}--(0.5,0)node[snode]{};
\coordinate[label=right:$v$] (A) at (0.5,0);
\end{tikzpicture}
\end{minipage}
\begin{minipage}{0.19\textwidth}
\centering
\begin{tikzpicture}
\draw(-0.5,0)node[bnode]{}--(0,0)node[bnode]{}--(0.5,0)node[bnode]{}--(1,0)node[snode]{};
\coordinate[label=above:$v$] (A) at (0,0);
\end{tikzpicture}
\end{minipage}
\begin{minipage}{0.19\textwidth}
\centering
\begin{tikzpicture}
\draw(0,0.25)node[bnode]{}--(0.5,0)node[snode]{}--(0,-0.25)node[bnode]{}--(-0.5,-0.25)node[bnode]{};
\coordinate[label=left:$v$] (A) at (-0.5,-0.25);
\end{tikzpicture}
\end{minipage}
\begin{minipage}{0.19\textwidth}
\centering
\begin{tikzpicture}
\draw(-0.5,0.25)node[bnode]{}--(0,0)node[snode]{}--(-0.5,-0.25)node[bnode]{}--(-1,-0.25)node[bnode]{}--(-1.5,-0.25)node[bnode]{};
\coordinate[label=right:$v$] (A) at (0,0);
\end{tikzpicture}
\end{minipage}
\\[1em]
\begin{minipage}{0.19\textwidth}
\centering $T_1$
\end{minipage}
\begin{minipage}{0.19\textwidth}
\centering $T_2$
\end{minipage}
\begin{minipage}{0.19\textwidth}
\centering $T_3$
\end{minipage}
\begin{minipage}{0.19\textwidth}
\centering $T_4$
\end{minipage}
\begin{minipage}{0.19\textwidth}
\centering $T_5$
\end{minipage}
\\[2em]
\begin{minipage}{0.24\textwidth}
\centering
\begin{tikzpicture}
\draw(0.5,-0.25)node[snode]{}--(0,-0.25)node[bnode]{}--(-0.5,-0.25)node[bnode]{}--(-1,-0.25)node[bnode]{}--(-1.5,-0.25)node[bnode]{};
\coordinate[label=right:${v=v_2}$](A) at (0.5,-0.25);
\coordinate[label=left:$v_1$] (A) at (-1.5,-0.25);
\end{tikzpicture}
\end{minipage}
\begin{minipage}{0.24\textwidth}
\centering
\begin{tikzpicture}
\draw(0.5,-0.25)node[bnode]{}--(1,-0.25)node[bnode]{}--(1.5,-0.25)node[bnode]{}--(2,-0.25)node[bnode]{}--(2.5,-0.25)node[bnode]{}--(3,-0.25)node[snode]{};
\coordinate[label=above:$v$] (A) at(2,-0.25);
\end{tikzpicture}
\end{minipage}
\begin{minipage}{0.24\textwidth}
\centering
\begin{tikzpicture}
\draw(-0.5,0.25)node[bnode]{}--(0,0)node[snode]{}--(-0.5,-0.25)node[bnode]{}--(-1,-0.25)node[bnode]{}--(-1.5,-0.25)node[bnode]{}--(-2,-0.25)node[bnode]{};
\coordinate[label=right:$v$] (A) at (0,0);
\end{tikzpicture}
\end{minipage}
\begin{minipage}{0.24\textwidth}
\centering
\begin{tikzpicture}
\draw(0.5,-0.25)node[snode]{}--(0,-0.25)node[bnode]{}--(-0.5,-0.25)node[bnode]{}--(-1,-0.25)node[bnode]{}--(-1.5,-0.25)node[bnode]{}--(-2,-0.25)node[bnode]{}--(-2.5,-0.25)node[bnode]{};
\coordinate[label=right:$v$] (A) at (0.5,-0.25);
\end{tikzpicture}
\end{minipage}
\\[1em]
\begin{minipage}{0.24\textwidth}
\centering $T_6$
\end{minipage}
\begin{minipage}{0.24\textwidth}
\centering $T_7$
\end{minipage}
\begin{minipage}{0.24\textwidth}
\centering $T_8$
\end{minipage}
\begin{minipage}{0.24\textwidth}
\centering $T_9$
\end{minipage}
\\[2em]
\begin{minipage}{0.24\textwidth}
\centering
\begin{tikzpicture}
\draw(-1.5,0.25)node[bnode]{}--(-1,0.25)node[bnode]{}--(-0.5,0.25)node[bnode]{}--(0,0)node[snode]{}--(-0.5,-0.25)node[bnode]{}--(-1,-0.25)node[bnode]{}--(-1.5,-0.25)node[bnode]{};
\coordinate[label=right:$v$] (A) at (0,0);
\end{tikzpicture}
\end{minipage}
\begin{minipage}{0.24\textwidth}
\centering
\begin{tikzpicture}
\draw(-0.5,0.25)node[bnode]{}--(0,0.25)node[snode]{}--(-0.5,-0.25)node[bnode]{}--(-1,-0.25)node[bnode]{}--(-1.5,-0.25)node[bnode]{}--(-2,-0.25)node[bnode]{};
\draw(-0.5,0.75)node[bnode]{}--(0,0.25)node[snode]{};
\coordinate[label=left:$v$] (A) at (-2,-0.25);
\end{tikzpicture}
\end{minipage}
\begin{minipage}{0.24\textwidth}
\centering
\begin{tikzpicture}
\draw(-0.5,-0.25)node[snode]{}--(-1,-0.25)node[bnode]{}--(-1.5,-0.25)node[bnode]{}--(-2,-0.25)node[bnode]{}--(-2.5,-0.25)node[bnode]{};
\draw(-1,-0.25)node[bnode]{}--(-1.5,0.25)node[bnode]{};
\draw(-.5,-0.25)node[snode]{}--(-1,0.25)node[bnode]{};
\coordinate[label=above:$v$] (A) at(-1.5,0.25);
\end{tikzpicture}
\end{minipage}
\begin{minipage}{0.24\textwidth}
\centering
\begin{tikzpicture}
\draw(-0.5,-0.25)node[snode]{}--(-1,-0.25)node[bnode]{}--(-1.5,-0.25)node[bnode]{}--(-2,-0.25)node[bnode]{}--(-2.5,-0.25)node[bnode]{};
\draw(-1,-0.25)node[bnode]{}--(-1.5,0.25)node[bnode]{};
\draw(-.5,-0.25)node[snode]{}--(-1,0.25)node[bnode]{};
\coordinate[label=above:$v$] (A) at(-1,0.25);
\end{tikzpicture}
\end{minipage}
\\[1em]
\begin{minipage}{0.24\textwidth}
\centering $T_{10}$
\end{minipage}
\begin{minipage}{0.24\textwidth}
\centering $T_{11}$
\end{minipage}
\begin{minipage}{0.24\textwidth}
\centering $T_{12}$
\end{minipage}
\begin{minipage}{0.24\textwidth}
\centering $T_{13}$
\end{minipage}
\\[2em]
\begin{minipage}{0.49\textwidth}
\centering
\begin{tikzpicture}
\draw(0,-0.25)node[snode]{}--(-0.5,-0.25)node[bnode]{}--(-1,-0.25)node[bnode]{}--(-1.5,-0.25)node[bnode]{}--(-2,-0.25)node[bnode]{}--(-2.5,-0.25)node[bnode]{};
\draw(-1,-0.25)node[bnode]{}--(-1.5,0.25)node[bnode]{};
\draw(-.5,-0.25)node[bnode]{}--(-1,0.25)node[bnode]{};
\coordinate[label=right:${v=v_1}$] (A) at(0,-0.25);
\coordinate[label=above:${v_2}$] (A) at(-0.5,-0.25);
\end{tikzpicture}
\end{minipage}
\begin{minipage}{0.5\textwidth}
\centering
\begin{tikzpicture}
\draw(-1,0.25)node[bnode]{}--(-0.5,0.25)node[bnode]{}--(0,0.25)node[bnode]{}--(0.5,0.25)node[snode]{};
\draw(-1,0.75)node[bnode]{}--(-0.5,.25)node[bnode]{};
\draw(-.5,0.75)node[bnode]{}--(0,.25)node[bnode]{};
\draw(0,0.75)node[bnode]{}--(0.5,.25)node[snode]{};
\draw(0,-0.25)node[bnode]{}--(0.5,.25)node[snode]{};
\coordinate[label=left:$v$] (A) at (-1,0.25);
\end{tikzpicture}
\end{minipage}
\\[1em]
\begin{minipage}{0.49\textwidth}
\centering $T_{14}$
\end{minipage}
\begin{minipage}{0.5\textwidth}
\centering $T_{15}$
\end{minipage}
\end{center}
\caption{The set $A = \{T_1, \ldots, T_{15}\}$ of special trees}\label{fig1A}
\end{figure}
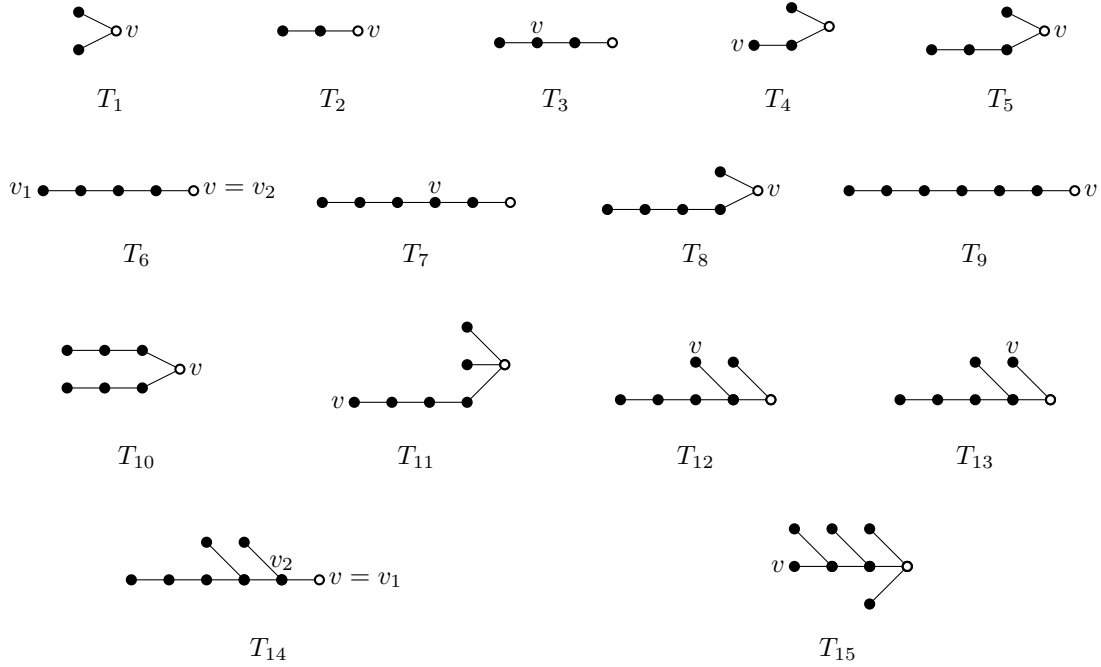
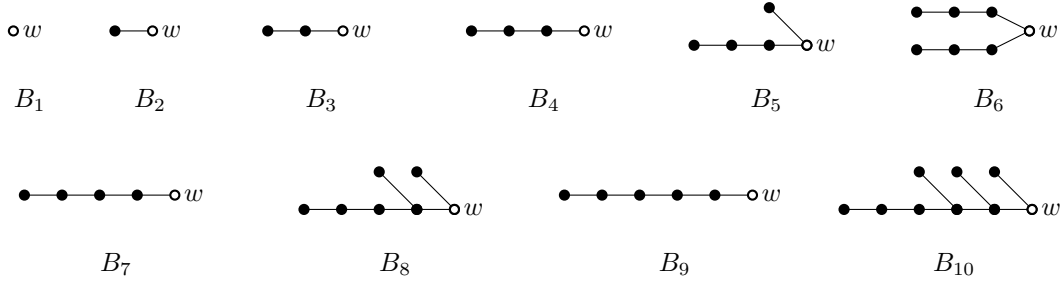
\begin{figure}[h]
\begin{center}
\begin{minipage}{0.10\textwidth}
\centering
\begin{tikzpicture}
\draw(0,0.25)node[snode]{};
\coordinate[label=right:$w$] (A) at (0,0.25);
\end{tikzpicture}
\end{minipage}
\begin{minipage}{0.10\textwidth}
\centering
\begin{tikzpicture}
\draw(-0.5,0.25)node[bnode]{}--(0,0.25)node[snode]{};
\coordinate[label=right:$w$] (A) at (0,0.25);
\end{tikzpicture}
\end{minipage}
\begin{minipage}{0.19\textwidth}
\centering
\begin{tikzpicture}
\draw(-1,0.25)node[bnode]{}--(-0.5,0.25)node[bnode]{}--(0,0.25)node[snode]{};
\coordinate[label=right:$w$] (A) at (0,0.25);
\end{tikzpicture}
\end{minipage}
\begin{minipage}{0.19\textwidth}
\centering
\begin{tikzpicture}
\draw(-2,0.25)node[bnode]{}--(-1.5,0.25)node[bnode]{}--(-1,0.25)node[bnode]{}--(-0.5,0.25)node[snode]{};
\coordinate[label=right:$w$] (A) at (-0.5,0.25);
\end{tikzpicture}
\end{minipage}
\begin{minipage}{0.19\textwidth}
\centering
\begin{tikzpicture}
\draw(-2,0.25)node[bnode]{}--(-1.5,0.25)node[bnode]{}--(-1,0.25)node[bnode]{}--(-0.5,0.25)node[snode]{};
\draw(-1,0.75)node[bnode]{}--(-0.5,.25)node[snode]{};
\coordinate[label=right:$w$] (A) at (-0.5,0.25);
\end{tikzpicture}
\end{minipage}
\begin{minipage}{0.19\textwidth}
\centering
\begin{tikzpicture}
\draw(-2,0.25)node[bnode]{}--(-1.5,0.25)node[bnode]{}--(-1,0.25)node[bnode]{}--(-0.5,0)node[snode]{};
\draw(-2,-0.25)node[bnode]{}--(-1.5,-0.25)node[bnode]{}--(-1,-0.25)node[bnode]{}--(-0.5,0)node[snode]{};
\coordinate[label=right:$w$] (A) at (-0.5,0);
\end{tikzpicture}
\end{minipage}
\\[1em]
\begin{minipage}{0.10\textwidth}
\centering
$B_1$
\centering
\end{minipage}
\begin{minipage}{0.10\textwidth}
\centering
$B_2$
\end{minipage}
\begin{minipage}{0.19\textwidth}
\centering
$B_3$
\end{minipage}
\begin{minipage}{0.19\textwidth}
\centering
$B_4$
\end{minipage}
\begin{minipage}{0.19\textwidth}
\centering
$B_5$
\end{minipage}
\begin{minipage}{0.19\textwidth}
\centering
$B_6$
\end{minipage}
\\[2em]
\begin{minipage}{0.24\textwidth}
\centering
\begin{tikzpicture}
\draw(-2,0.25)node[bnode]{}--(-1.5,0.25)node[bnode]{}--(-1,0.25)node[bnode]{}--(-0.5,0.25)node[bnode]{}--(0,0.25)node[snode]{};
\coordinate[label=right:$w$] (A) at (0,0.25);
\end{tikzpicture}
\end{minipage}
\begin{minipage}{0.24\textwidth}
\centering
\begin{tikzpicture}
\draw(-2,0.25)node[bnode]{}--(-1.5,0.25)node[bnode]{}--(-1,0.25)node[bnode]{}--(-0.5,0.25)node[bnode]{}--(0,0.25)node[snode]{};
\draw(-1,0.75)node[bnode]{}--(-0.5,.25)node[bnode]{};
\draw(-.5,0.75)node[bnode]{}--(0,.25)node[snode]{};
\coordinate[label=right:$w$] (A) at (0,0.25);
\end{tikzpicture}
\end{minipage}
\begin{minipage}{0.24\textwidth}
\centering
\begin{tikzpicture}
\draw(-2.5,0.25)node[bnode]{}--(-2,0.25)node[bnode]{}--(-1.5,0.25)node[bnode]{}--(-1,0.25)node[bnode]{}--(-0.5,0.25)node[bnode]{}--(0,0.25)node[snode]{};
\coordinate[label=right:$w$] (A) at (0,0.25);
\end{tikzpicture}
\end{minipage}
\begin{minipage}{0.24\textwidth}
\centering
\begin{tikzpicture}
\draw(-2,0.25)node[bnode]{}--(-1.5,0.25)node[bnode]{}--(-1,0.25)node[bnode]{}--(-0.5,0.25)node[bnode]{}--(0,0.25)node[bnode]{}--(0.5,0.25)node[snode]{};
\draw(-1,0.75)node[bnode]{}--(-0.5,.25)node[bnode]{};
\draw(-.5,0.75)node[bnode]{}--(0,.25)node[bnode]{};
\draw(0,0.75)node[bnode]{}--(0.5,.25)node[snode]{};
\coordinate[label=right:$w$] (A) at (0.5,0.25);
\end{tikzpicture}
\end{minipage}
\\[1em]
\begin{minipage}{0.24\textwidth}
\centering $B_{7}$
\end{minipage}
\begin{minipage}{0.24\textwidth}
\centering $B_{8}$
\end{minipage}
\begin{minipage}{0.24\textwidth}
\centering $B_{9}$
\end{minipage}
\begin{minipage}{0.24\textwidth}
\centering $B_{10}$
\end{minipage}
\end{center}
\caption{The set $B = \{B_1, \ldots, B_{10}\}$ of special trees}\label{fig1B}
\end{figure}

For each special tree, we $2$-color the vertices with the colors white and black as illustrated in Figures 1 and 2 to indicate the roles they play in the tree. We note that exactly one vertex in each special tree is white. Given a special tree $T$, we denote the set of black vertices by $V_B(T)$. We also specify certain vertices of each special tree $T$, which we name $v(T)$, $v_1(T)$, $v_2(T)$ and $w(T)$. If a special tree $T$ is clear from context, then we simply refer to these specified vertices as $v$, $v_1$, $v_2$, and $w$. We remark that some special trees occur more than once in Figures 1 and 2. However, for simplicity in the proofs that follow, we assign different names to these special trees.

Let $T_{\pdi} \in A \cup B$ be a special tree and let $T$ be a tree. If $T$ contains a subset $U$ of vertices such that $T[U] \cong T_{\pdi}$ and the degree of every black vertex in $V_B(T_{\pdi})$ equals its degree in $T$, then we say that the tree $T$ contains $T_{\pdi}$ as a {\it prescribed-degree-induced subtree}, abbreviated {\it PDI-subtree}. In particular, we note that if $T_{\pdi}$ is a PDI-subtree of a tree $T$, then the degree sequence of the vertices of $V_B(T_{\pdi})$ in $T$ equals the degree sequence of the vertices of $V_B(T_{\pdi})$ in $T_{\pdi}$.

We are now in position to define our family~$\cT$.
\begin{definition}\label{defn1}
Let $\cT$ be the family of trees that:
\begin{enumerate}
\item contains all trees of order at most~$4$,\\ [-1.75em]
\item is closed under the four Operations $\cO_1$, $\cO_2$, $\cO_3$, and $\cO_4$ that are listed below, which extend the tree $T'$ to a tree $T$ by attaching a tree to the vertex $v \in V(T')$, called the \emph{attacher} of~$T'$, and\\ [-1.75em]
\item is closed under the Operations $\cO_5$, and $\cO_6$ listed below, which extend the tree $T'$ to a tree $T$ by attaching trees to the vertices $v_1$ and $v_2$ of $T'$, called the \emph{attachers} of~$T'$.\\ [-1.75em]
\end{enumerate}
\begin{itemize}
\item \textbf{Operation $\cO_1$}: Let $T_{\pdi} \in \{T_1,T_2,T_8\}$ be a PDI-subtree of $T'$ and let $v = v(T_{\pdi})$. Add a new vertex $u$ and the edge $vu$.\\ [-1.75em]
\item \textbf{Operation $\cO_2$}: Let $T_{\pdi} \in \{T_4,T_{11},T_{12},T_{13},T_{15}\}$ be a PDI-subtree of $T'$ and let $v = v(T_{\pdi})$. Add a path $u_1u_2$ to $T'$ and the edge $vu_1$. \\ [-1.75em]
\item \textbf{Operation $\cO_3$}: Let $v$ be an arbitrary vertex of $T'$. Add a path $u_1u_2u_3$ to $T'$ and the edge $vu_2$.\\ [-1.75em]
\item \textbf{Operation $\cO_4$}: Let $T_{\pdi} \in \{T_1,T_2,T_3,T_5,T_6,T_7,T_9,T_{10}\}$ be a PDI-subtree of $T'$ and let $v = v(T_{\pdi})$. Add a path $u_1u_2u_3$ to $T'$ and the edge $vu_1$.\\ [-1.75em]
\item \textbf{Operation $\cO_5$}: Let $T_{\pdi} \cong T_6$ be a PDI-subtree of $T'$ and let $v_1 = v_1(T_{\pdi})$, $v_2 = v_2(T_{\pdi})$. Add a path $u_1u_3$ to $T'$ and the edge $v_1u_1$, and add a new vertex $u_2$ and the edge $v_2u_2$.\\ [-1.75em]
\item \textbf{Operation $\cO_6$}: Let $T_{\pdi} \cong T_{14}$ be a PDI-subtree of $T'$ and let $v_1 = v_1(T_{\pdi})$, $v_2 = v_2(T_{\pdi})$. Remove the edge $v_1v_2$, and add a path $u_1u_2u_3$ and the edges $v_1u_1$ and $v_2u_2$.
\end{itemize}
\end{definition}

For $i \in \{1,2,4,5,6\}$, if $T$ is obtained from $T'$ by applying Operation $\mathcal{O}_i$ to a PDI-subtree $T_{\pdi}$ of $T'$, then we let $X = V(T) \setminus V(T')$ and $T_{\pdi}^{\mathcal{O}_i} = T[V(T_{\pdi}) \cup X]$. Further, we color all vertices of $X$ in $T_{\pdi}^{\mathcal{O}_i}$ black, while the colors of all vertices in the set $V(T_{\pdi})$ remain unchanged.

We shall need the following properties of special trees.

\begin{obs}\label{ob:PDI}
If $T$ is a special tree, then the vertices of $T$ covered by a square in Figures 3 and 4 form a $\gamma_2(T)$-set, and the vertices covered by a diamond form an $\alpha_2(T)$-set.
\end{obs}

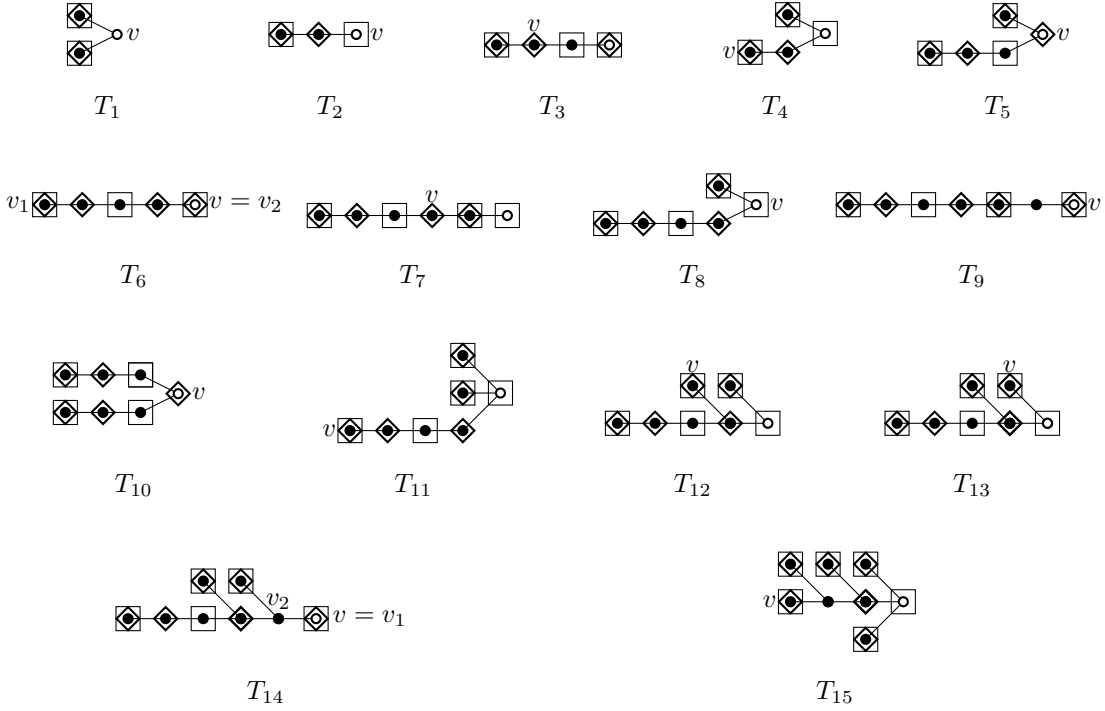
\begin{figure}[h]
\begin{center}
\begin{minipage}{0.19\textwidth}
\centering
\begin{tikzpicture}
\draw(0,0.25)node[bnode]{}node[alpha2]{}node[gamma2]{}--(0.5,0)node[snode]{}--(0,-0.25)node[alpha2]{}node[gamma2]{}node[bnode]{};
\coordinate[label=right:$v$] (A) at (0.5,0);
\end{tikzpicture}
\end{minipage}
\begin{minipage}{0.19\textwidth}
\centering
\begin{tikzpicture}
\draw(-0.5,0)node[alpha2]{}node[gamma2]{}node[bnode]{}--(0,0)node[alpha2]{}node[bnode]{}--(0.5,0)node[snode]{}node[gamma2]{};
\coordinate[label=right:$v$] (A) at (0.56,0);
\end{tikzpicture}
\end{minipage}
\begin{minipage}{0.19\textwidth}
\centering
\begin{tikzpicture}
\draw(-0.5,0)node[alpha2]{}node[gamma2]{}node[bnode]{}--(0,0)node[alpha2]{}node[bnode]{}--(0.5,0)node[bnode]{}node[gamma2]{}--(1,0)node[snode]{}node[alpha2]{}node[gamma2]{};
\coordinate[label=above:$v$] (A) at (0,0.06);
\end{tikzpicture}
\end{minipage}
\begin{minipage}{0.19\textwidth}
\centering
\begin{tikzpicture}
\draw(0,0.25)node[bnode]{}node[alpha2]{}node[gamma2]{}--(0.5,0)node[snode]{}node[gamma2]{}--(0,-0.25)node[bnode]{}node[alpha2]{}--(-0.5,-0.25)node[bnode]{}node[alpha2]{}node[gamma2]{};
\coordinate[label=left:$v$] (A) at (-0.56,-0.25);
\end{tikzpicture}
\end{minipage}
\begin{minipage}{0.19\textwidth}
\centering
\begin{tikzpicture}
\draw(-0.5,0.25)node[bnode]{}node[alpha2]{}node[gamma2]{}--(0,0)node[snode]{}node[alpha2]{}--(-0.5,-0.25)node[bnode]{}node[gamma2]{}--(-1,-0.25)node[bnode]{}node[alpha2]{}--(-1.5,-0.25)node[bnode]{}node[alpha2]{}node[gamma2]{};
\coordinate[label=right:$v$] (A) at (0.06,0);
\end{tikzpicture}
\end{minipage}
\\[1em]
\begin{minipage}{0.19\textwidth}
\centering $T_1$
\end{minipage}
\begin{minipage}{0.19\textwidth}
\centering $T_2$
\end{minipage}
\begin{minipage}{0.19\textwidth}
\centering $T_3$
\end{minipage}
\begin{minipage}{0.19\textwidth}
\centering $T_4$
\end{minipage}
\begin{minipage}{0.19\textwidth}
\centering $T_5$
\end{minipage}
\\[2em]
\begin{minipage}{0.24\textwidth}
\centering
\begin{tikzpicture}
\draw(0.5,-0.25)node[snode]{}node[alpha2]{}node[gamma2]{}--(0,-0.25)node[bnode]{}node[alpha2]{}--(-0.5,-0.25)node[bnode]{}node[gamma2]{}--(-1,-0.25)node[bnode]{}node[alpha2]{}--(-1.5,-0.25)node[bnode]{}node[alpha2]{}node[gamma2]{};
\coordinate[label=right:${v=v_2}$](A) at (0.56,-0.25);
\coordinate[label=left:$v_1$] (A) at (-1.56,-0.25);
\end{tikzpicture}
\end{minipage}
\begin{minipage}{0.24\textwidth}
\centering
\begin{tikzpicture}
\draw(0.5,-0.25)node[bnode]{}node[alpha2]{}node[gamma2]{}--(1,-0.25)node[bnode]{}node[alpha2]{}--(1.5,-0.25)node[bnode]{}node[gamma2]{}--(2,-0.25)node[bnode]{}node[alpha2]{}--(2.5,-0.25)node[bnode]{}node[alpha2]{}node[gamma2]{}--(3,-0.25)node[snode]{}node[gamma2]{};
\coordinate[label=above:$v$] (A) at(2.,-0.19);
\end{tikzpicture}
\end{minipage}
\begin{minipage}{0.24\textwidth}
\centering
\begin{tikzpicture}
\draw(-0.5,0.25)node[bnode]{}node[alpha2]{}node[gamma2]{}--(0,0)node[snode]{}node[gamma2]{}--(-0.5,-0.25)node[bnode]{}node[alpha2]{}--(-1,-0.25)node[bnode]{}node[gamma2]{}--(-1.5,-0.25)node[bnode]{}node[alpha2]{}--(-2,-0.25)node[bnode]{}node[alpha2]{}node[gamma2]{};
\coordinate[label=right:$v$] (A) at (0.06,0);
\end{tikzpicture}
\end{minipage}
\begin{minipage}{0.24\textwidth}
\centering
\begin{tikzpicture}
\draw(0.5,-0.25)node[snode]{}node[alpha2]{}node[gamma2]{}--(0,-0.25)node[bnode]{}--(-0.5,-0.25)node[bnode]{}node[alpha2]{}node[gamma2]{}--(-1,-0.25)node[bnode]{}node[alpha2]{}--(-1.5,-0.25)node[bnode]{}node[gamma2]{}--(-2,-0.25)node[bnode]{}node[alpha2]{}--(-2.5,-0.25)node[bnode]{}node[alpha2]{}node[gamma2]{};
\coordinate[label=right:$v$] (A) at (0.56,-0.25);
\end{tikzpicture}
\end{minipage}
\\[1em]
\begin{minipage}{0.24\textwidth}
\centering $T_6$
\end{minipage}
\begin{minipage}{0.24\textwidth}
\centering $T_7$
\end{minipage}
\begin{minipage}{0.24\textwidth}
\centering $T_8$
\end{minipage}
\begin{minipage}{0.24\textwidth}
\centering $T_9$
\end{minipage}
\\[2em]
\begin{minipage}{0.24\textwidth}
\centering
\begin{tikzpicture}
\draw(-1.5,0.25)node[bnode]{}node[alpha2]{}node[gamma2]{}--(-1,0.25)node[bnode]{}node[alpha2]{}--(-0.5,0.25)node[bnode]{}node[gamma2]{}node[gamma2]{}--(0,0)node[snode]{}node[alpha2]{}--(-0.5,-0.25)node[bnode]{}node[gamma2]{}--(-1,-0.25)node[bnode]{}node[alpha2]{}--(-1.5,-0.25)node[bnode]{}node[alpha2]{}node[gamma2]{};
\coordinate[label=right:$v$] (A) at (0.06,0);
\end{tikzpicture}
\end{minipage}
\begin{minipage}{0.24\textwidth}
\centering
\begin{tikzpicture}
\draw(-0.5,0.25)node[bnode]{}node[alpha2]{}node[gamma2]{}--(0,0.25)node[gamma2]{}--(-0.5,-0.25)node[bnode]{}node[alpha2]{}--(-1,-0.25)node[bnode]{}node[gamma2]{}--(-1.5,-0.25)node[bnode]{}node[alpha2]{}--(-2,-0.25)node[bnode]{}node[alpha2]{}node[gamma2]{};
\draw(-0.5,0.75)node[alpha2]{}node[bnode]{}node[gamma2]{}--(0,0.25)node[snode]{};
\coordinate[label=left:$v$] (A) at (-2.06,-0.25);
\end{tikzpicture}
\end{minipage}
\begin{minipage}{0.24\textwidth}
\centering
\begin{tikzpicture}
\draw(-0.5,-0.25)node[gamma2]{}--(-1,-0.25)node[bnode]{}--(-1.5,-0.25)node[bnode]{}node[gamma2]{}--(-2,-0.25)node[bnode]{}node[alpha2]{}--(-2.5,-0.25)node[bnode]{}node[alpha2]{}node[gamma2]{};
\draw(-1,-0.25)node[bnode]{}node[alpha2]{}--(-1.5,0.25)node[bnode]{}node[alpha2]{}node[gamma2]{};
\draw(-.5,-0.25)node[snode]{}--(-1,0.25)node[bnode]{}node[alpha2]{}node[gamma2]{};
\coordinate[label=above:$v$] (A) at(-1.5,0.31);
\end{tikzpicture}
\end{minipage}
\begin{minipage}{0.24\textwidth}
\centering
\begin{tikzpicture}
\draw(-0.5,-0.25)node[snode]{}--(-1,-0.25)node[bnode]{}node[alpha2]{}--(-1.5,-0.25)node[bnode]{}node[gamma2]{}--(-2,-0.25)node[bnode]{}node[alpha2]{}--(-2.5,-0.25)node[bnode]{}node[alpha2]{}node[gamma2]{};
\draw(-1,-0.25)node[bnode]{}node[alpha2]{}--(-1.5,0.25)node[bnode]{}node[alpha2]{}node[gamma2]{};
\draw(-.5,-0.25)node[snode]{}node[gamma2]{}--(-1,0.25)node[bnode]{}node[alpha2]{}node[gamma2]{};
\coordinate[label=above:$v$] (A) at(-1,0.31);
\end{tikzpicture}
\end{minipage}
\\[1em]
\begin{minipage}{0.24\textwidth}
\centering $T_{10}$
\end{minipage}
\begin{minipage}{0.24\textwidth}
\centering $T_{11}$
\end{minipage}
\begin{minipage}{0.24\textwidth}
\centering $T_{12}$
\end{minipage}
\begin{minipage}{0.24\textwidth}
\centering $T_{13}$
\end{minipage}
\\[2em]
\begin{minipage}{0.49\textwidth}
\centering
\begin{tikzpicture}
\draw(0,-0.25)node[snode]{}node[alpha2]{}node[gamma2]{}--(-0.5,-0.25)node[bnode]{}--(-1,-0.25)node[bnode]{}node[alpha2]{}--(-1.5,-0.25)node[bnode]{}node[gamma2]{}--(-2,-0.25)node[bnode]{}node[alpha2]{}--(-2.5,-0.25)node[bnode]{}node[alpha2]{}node[gamma2]{};
\draw(-1,-0.25)node[bnode]{}node[alpha2]{}--(-1.5,0.25)node[bnode]{}node[alpha2]{}node[gamma2]{};
\draw(-.5,-0.25)node[bnode]{}--(-1,0.25)node[bnode]{}node[alpha2]{}node[gamma2]{};
\coordinate[label=right:${v=v_1}$] (A) at(0.1,-0.25);
\coordinate[label=above:${v_2}$] (A) at(-0.5,-0.25);
\end{tikzpicture}
\end{minipage}
\begin{minipage}{0.5\textwidth}
\centering
\begin{tikzpicture}
\draw(-1,0.25)node[bnode]{}node[alpha2]{}node[gamma2]{}--(-0.5,0.25)node[bnode]{}--(0,0.25)node[bnode]{}node[alpha2]{}--(0.5,0.25)node[gamma2]{};
\draw(-1,0.75)node[bnode]{}node[alpha2]{}node[gamma2]{}--(-0.5,.25)node[bnode]{};
\draw(-.5,0.75)node[bnode]{}node[alpha2]{}node[gamma2]{}--(0,.25)node[bnode]{}node[alpha2]{};
\draw(0,0.75)node[bnode]{}node[alpha2]{}node[gamma2]{}--(0.5,.25);
\draw(0,-0.25)node[bnode]{}node[alpha2]{}node[gamma2]{}--(0.5,.25)node[snode]{};
\coordinate[label=left:$v$] (A) at (-1.06,0.25);
\end{tikzpicture}
\end{minipage}
\\[1em]
\begin{minipage}{0.49\textwidth}
\centering $T_{14}$
\end{minipage}
\begin{minipage}{0.5\textwidth}
\centering $T_{15}$
\end{minipage}
\end{center}
\caption{The set $A = \{T_1,\ldots,T_{15}\}$ of special trees}\label{fig2A}
\end{figure}

\begin{figure}[h]
\begin{center}
\begin{minipage}{0.10\textwidth}
\centering
\begin{tikzpicture}
\draw(0,0.25)node[snode]{}node[alpha2]{}node[gamma2]{};
\coordinate[label=right:$w$] (A) at (0.06,0.25);
\end{tikzpicture}
\end{minipage}
\begin{minipage}{0.10\textwidth}
\centering
\begin{tikzpicture}
\draw(-0.5,0.25)node[bnode]{}node[alpha2]{}node[gamma2]{}--(0,0.25)node[snode]{}node[alpha2]{}node[gamma2]{};
\coordinate[label=right:$w$] (A) at (0.06,0.25);
\end{tikzpicture}
\end{minipage}
\begin{minipage}{0.19\textwidth}
\centering
\begin{tikzpicture}
\draw(-1,0.25)node[bnode]{}node[alpha2]{}node[gamma2]{}--(-0.5,0.25)node[bnode]{}node[alpha2]{}--(0,0.25)node[snode]{}node[gamma2]{};
\coordinate[label=right:$w$] (A) at (0.06,0.25);
\end{tikzpicture}
\end{minipage}
\begin{minipage}{0.19\textwidth}
\centering
\begin{tikzpicture}
\draw(-2,0.25)node[bnode]{}node[alpha2]{}node[gamma2]{}--(-1.5,0.25)node[bnode]{}node[alpha2]{}--(-1,0.25)node[bnode]{}node[gamma2]{}--(-0.5,0.25)node[snode]{}node[alpha2]{}node[gamma2]{};
\coordinate[label=right:$w$] (A) at (-0.44,0.25);
\end{tikzpicture}
\end{minipage}
\begin{minipage}{0.19\textwidth}
\centering
\begin{tikzpicture}
\draw(-2,0.25)node[bnode]{}node[alpha2]{}node[gamma2]{}--(-1.5,0.25)node[bnode]{}node[alpha2]{}--(-1,0.25)node[bnode]{}node[gamma2]{}--(-0.5,0.25)node[alpha2]{};
\draw(-1,0.75)node[bnode]{}node[alpha2]{}node[gamma2]{}--(-0.5,.25)node[snode]{}node[alpha2]{};
\coordinate[label=right:$w$] (A) at (-0.44,0.25);
\end{tikzpicture}
\end{minipage}
\begin{minipage}{0.19\textwidth}
\centering
\begin{tikzpicture}
\draw(-2,0.25)node[bnode]{}node[alpha2]{}node[gamma2]{}--(-1.5,0.25)node[bnode]{}node[alpha2]{}--(-1,0.25)node[bnode]{}node[gamma2]{}--(-0.5,0)node[alpha2]{};
\draw(-2,-0.25)node[bnode]{}node[alpha2]{}node[gamma2]{}--(-1.5,-0.25)node[bnode]{}node[alpha2]{}--(-1,-0.25)node[bnode]{}node[gamma2]{}--(-0.5,0)node[snode]{};
\coordinate[label=right:$w$] (A) at (-0.44,0);
\end{tikzpicture}
\end{minipage}
\\[1em]
\begin{minipage}{0.10\textwidth}
\centering
$B_1$
\centering
\end{minipage}
\begin{minipage}{0.10\textwidth}
\centering
$B_2$
\end{minipage}
\begin{minipage}{0.19\textwidth}
\centering
$B_3$
\end{minipage}
\begin{minipage}{0.19\textwidth}
\centering
$B_4$
\end{minipage}
\begin{minipage}{0.19\textwidth}
\centering
$B_5$
\end{minipage}
\begin{minipage}{0.19\textwidth}
\centering
$B_6$
\end{minipage}
\\[2em]
\begin{minipage}{0.24\textwidth}
\centering
\begin{tikzpicture}
\draw(-2,0.25)node[bnode]{}node[alpha2]{}node[gamma2]{}--(-1.5,0.25)node[bnode]{}node[alpha2]{}--(-1,0.25)node[bnode]{}node[gamma2]{}--(-0.5,0.25)node[bnode]{}node[alpha2]{}--(0,0.25)node[snode]{}node[gamma2]{};
\coordinate[label=right:$w$] (A) at (0.06,0.25);
\end{tikzpicture}
\end{minipage}
\begin{minipage}{0.24\textwidth}
\centering
\begin{tikzpicture}
\draw(-2,0.25)node[bnode]{}node[alpha2]{}node[gamma2]{}--(-1.5,0.25)node[bnode]{}node[alpha2]{}--(-1,0.25)node[bnode]{}node[gamma2]{}--(-0.5,0.25)node[bnode]{}node[alpha2]{}--(0,0.25);
\draw(-1,0.75)node[bnode]{}node[alpha2]{}node[gamma2]{}--(-0.5,.25);
\draw(-.5,0.75)node[bnode]{}node[alpha2]{}node[gamma2]{}--(0,.25)node[snode]{}node[gamma2]{};
\coordinate[label=right:$w$] (A) at (0.06,0.25);
\end{tikzpicture}
\end{minipage}
\begin{minipage}{0.24\textwidth}
\centering
\begin{tikzpicture}
\draw(-2.5,0.25)node[bnode]{}node[alpha2]{}node[gamma2]{}--(-2,0.25)node[bnode]{}node[alpha2]{}--(-1.5,0.25)node[bnode]{}node[gamma2]{}--(-1,0.25)node[bnode]{}node[alpha2]{}--(-0.5,0.25)node[bnode]{}node[alpha2]{}node[gamma2]{}--(0,0.25)node[snode]{}node[gamma2]{};
\coordinate[label=right:$w$] (A) at (0.06,0.25);
\end{tikzpicture}
\end{minipage}
\begin{minipage}{0.24\textwidth}
\centering
\begin{tikzpicture}
\draw(-2,0.25)node[bnode]{}node[alpha2]{}node[gamma2]{}--(-1.5,0.25)node[bnode]{}node[alpha2]{}--(-1,0.25)node[bnode]{}node[gamma2]{}--(-0.5,0.25)node[bnode]{}node[alpha2]{}--(0,0.25)node[bnode]{}--(0.5,0.25)node[gamma2]{};
\draw(-1,0.75)node[bnode]{}node[alpha2]{}node[gamma2]{}--(-0.5,.25);
\draw(-.5,0.75)node[bnode]{}node[alpha2]{}node[gamma2]{}--(0,.25)node[bnode]{};
\draw(0,0.75)node[bnode]{}node[alpha2]{}node[gamma2]{}--(0.5,.25)node[snode]{}node[alpha2]{};
\coordinate[label=right:$w$] (A) at (0.56,0.25);
\end{tikzpicture}
\end{minipage}
\\[1em]
\begin{minipage}{0.24\textwidth}
\centering $B_{7}$
\end{minipage}
\begin{minipage}{0.24\textwidth}
\centering $B_{8}$
\end{minipage}
\begin{minipage}{0.24\textwidth}
\centering $B_{9}$
\end{minipage}
\begin{minipage}{0.24\textwidth}
\centering $B_{10}$
\end{minipage}
\end{center}
\caption{The set $B = \{B_1, \ldots, B_{10}\}$ of special trees}\label{fig2B}
\end{figure}
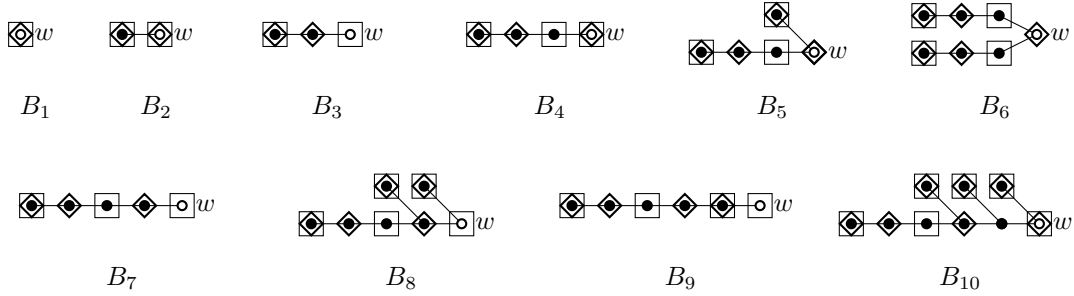

The following observation follows readily from the facts that in a rooted tree one can construct a minimum $2$-dominating set by ``pushing'' vertices in the direction of the root, in the sense that if we can replace a vertex in a $2$-dominating set by its parent, then we do so; further, we can construct a maximum $2$-independent set by ``pushing'' vertices away from the root as far as possible, in the sense that if we can replace a vertex in a $2$-independent set by its children, then we do so.

\begin{obs}\label{ob:PDIb}
Let $T'$ be a tree that contains a PDI-subtree $T_{\pdi}$, and let $D'$ be a $\gamma_2(T')$-set and $S'$ an $\alpha_2(T')$-set. Let $D_{\pdi} = D' \cap V(T_{\pdi})$ and $D_{\pdi}^B = D' \cap V_B(T_{\pdi})$, and let $S_{\pdi} = S' \cap V(T_{\pdi})$ and $S_{\pdi}^B = S' \cap V_B(T_{\pdi})$. Then the sets $D'$ and $S'$ can be chosen so that the following hold.
\begin{enumerate}
\item[{\rm (a)}]
If $T_{\pdi} \in \{T_1,T_2,T_4,T_8,T_{11},T_{15},B_3,B_7\}$, then the sets $D_{\pdi}$ and $S_{\pdi}$ consist of the square and diamond vertices, respectively, of $T_{\pdi}$ illustrated in Figures 3 and 4.
\item[{\rm (b)}]
If $T_{\pdi} \in \{T_5,T_6,T_9,T_{14},B_1\}$, then the sets $D_{\pdi}$ and $S_{\pdi}^B$ consist of the square and diamond vertices, respectively, of $T_{\pdi}$ illustrated in Figures 3 and 4.
\item[{\rm (c)}]
If $T_{\pdi} \in \{T_{12},T_{13},B_8,B_9,B_{10}\}$, then the sets $D_{\pdi}^B$ and $S_{\pdi}$ consist of the square and diamond vertices, respectively, of $T_{\pdi}$ illustrated in Figures 3 and 4.
\item[{\rm (d)}]
If $T_{\pdi} \in \{T_{10},B_2,B_4,B_5,B_6\}$, then the sets $D_{\pdi}^B$ and $S_{\pdi}^B$ consist of the square and diamond vertices, respectively, of $T_{\pdi}$ illustrated in Figures 3 and 4.
\item[{\rm (e)}]
If $T_{\pdi} \in \{T_3,T_4,T_7,T_{11},T_{12},T_{13},T_{15}\}$ and $v = v(T_{\pdi})$, then $\alpha_2(T_{\pdi}-v) = \alpha_2(T_{\pdi})-1$.
\end{enumerate}
\end{obs}

\section{Proof of Theorem \ref{t:tree}}\label{S:mainp}

In this section, we will prove our main result, namely Theorem \ref{t:tree}.  We first will present some preliminary results that we will need for our proof.

\begin{obs}\label{obs1}
Every leaf of a graph $G$ is in every $\gamma_2(G)$-set and there is an $\alpha_2(G)$-set containing all leafs of $G$.
\end{obs}

We now prove that performing Operations $\mathcal{O}_1,\mathcal{O}_2,\ldots,\mathcal{O}_6$ maintains the difference between the $2$-domination and the $2$-independence numbers.

\begin{lem}\label{lemma1}
If $T$ is obtained from an arbitrary tree $T'$ by applying one of the Operations $\mathcal{O}_i$ for some $i \in [6]$, then $\alpha_2(T)-\gamma_2(T) = \alpha_2(T')-\gamma_2(T')$.
\end{lem}
\begin{proof}
Let $c = \alpha_2(T')-\gamma_2(T')$. By Theorem \ref{t:relate}, we have $c \ge 0$. We prove that $\alpha_2(T)-\gamma_2(T) = c$. Let $D'$ be a $\gamma_2(T')$-set and let $S'$ be an $\alpha_2(T')$-set. Let $v$ be the attacher in $T'$ if $T$ is obtained from $T'$ using Operation $\mathcal{O}_1$, $\mathcal{O}_2$, $\mathcal{O}_3$, or $\mathcal{O}_4$, and let $v_1$ and $v_2$ be the attachers in $T'$ if $T$ is obtained from $T'$ using Operation $\mathcal{O}_5$ or $\mathcal{O}_6$. If $T$ is obtained from $T'$ by Operation $\mathcal{O}_1$, $\mathcal{O}_2$, $\mathcal{O}_4$, $\mathcal{O}_5$, or $\mathcal{O}_6$, then let $T_{\pdi}$ be the PDI-subtree of $T'$ used to construct the tree $T$, where $v = v(T_{\pdi})$, $v_1 = v_1(T_{\pdi})$ and $v_2 = v_2(T_{\pdi})$. Further, let $D_{\pdi} = D' \cap V(T_{\pdi})$ and $D_{\pdi}^B = D' \cap V_B(T_{\pdi})$, and let $S_{\pdi} = S' \cap V(T_{\pdi})$ and $S_{\pdi}^B = S' \cap V_B(T_{\pdi})$. By Observation \ref{ob:PDIb}, the sets $D'$ and $S'$ can be chosen so that properties (a)-(e) in the statement of the observation hold. Let $D$ be a $\gamma_2(T)$-set and let $S$ be an $\alpha_2(T)$-set. We consider six cases, depending on the operation applied to $T'$ in order to obtain the tree $T$. In all
cases, we show that $\alpha_2(T)-\gamma_2(T) = c$.

\textit{Case 1. $T$ is obtained from $T'$ by Operation $\mathcal{O}_1$.}
In this case, $T_{\pdi} \in \{T_1,T_2,T_8\}$. Let $u$ be the vertex added to $T'$ and $uv$ be the edge added to $T'$ to obtain $T$. By Observation \ref{obs1}, we have $u \in D$. By Observation \ref{ob:PDIb}(a), the sets $D$ and $S$ can be chosen so that $D \cap V(T_{\pdi})$ and $S \cap V(T_{\pdi})$ are the sets of square and diamond vertices, respectively, of $T_{\pdi}$ illustrated in Figure 3, noting that in this case the vertex $v = v(T_{\pdi})$ is the white vertex of $T_{\pdi}$. This implies that either $v \in D$, or $v \notin D$ and $v$ is dominated twice by the vertices of $D \setminus \{u\}$. In both cases, the set $D \setminus \{u\}$ is a $2$-dominating set of $T'$. Therefore, $\gamma_2(T') \le \gamma_2(T)-1$. It is easy to observe that every $2$-dominating set of $T'$ can be extended to a $2$-dominating set of $T$ by adding to it the vertex $u$, implying that $\gamma_2(T) \le \gamma_2(T')+1$. Consequently, $\gamma_2(T) = \gamma_2(T')+1$. Further, we note that $v \notin S$. Therefore $u \in S$, and $S \setminus \{u\}$ is a $2$-independent set of $T'$ implying that $\alpha_2(T') \ge \alpha_2(T)-1$. By Observation \ref{ob:PDIb}(a) we note that the vertex $v$ does not belong to the $\alpha_2(T')$-set $S'$. Thus, $S'$ can be extended to a $2$-independent set of $T$ by adding to it the vertex $u$, implying that $\alpha_2(T) \ge \alpha_2(T')+1$. Consequently, $\alpha_2(T) = \alpha_2(T')+1$. Thus, $\alpha_2(T)-\gamma_2(T) = \alpha_2(T')-\gamma_2(T') = c$.

\textit{Case 2. $T$ is obtained from $T'$ by Operation $\mathcal{O}_2$.}
In this case, $T_{\pdi} \in \{T_4,T_{11},T_{12},T_{13},T_{15}\}$. Let $u_1u_2$ be the path added to $T'$ and $vu_1$ the edge added to $T'$ to obtain $T$. Since $u_3$ is a leaf of $T$, we note that $u_3 \in D$. By Observation \ref{ob:PDIb}(a) and \ref{ob:PDIb}(c), both sets $D'$ and $S'$ contain the vertex $v = v(T_{\pdi})$. The set $D' \cup \{u_2\}$ and $S' \cup \{u_2\}$ are therefore $2$-dominating and $2$-independent sets, respectively, of $T$, implying that $\gamma_2(T) \le |D'|+1 = \gamma_2(T')+1$ and $\alpha_2(T) \ge |S'|+1 = \alpha_2(T')+1$. We now consider the sets $D$ and $S$. Necessarily, $u_2 \in D$. If $u_1 \in D$, then we can replace it with the vertex $v$. If $u_1 \notin D$, then $v \in D$ in order to dominate $u_1$ twice. Hence, we may choose $D$ so that $D \cap \{v,u_1,u_2\} = \{v,u_2\}$. Therefore, $D \setminus \{u_2\}$ is a $2$-dominating set in $T'$, and so $\gamma_2(T') \le \gamma_2(T)-1$. Consequently, $\gamma_2(T) = \gamma_2(T')+1$. We can always choose $S$ so that $u_2 \in S$. If $u_1 \notin S$, then $S \setminus \{u_1\}$ is a $2$-independent set in $T'$, implying that $\alpha_2(T') \ge |S|-1 = \alpha_2(T)-1$. Suppose that $u_1 \in S$. Then, $v \notin S$, and so $S \setminus \{u_1,u_2\}$ is a $2$-independent set of $T'-v$. By Observation \ref{ob:PDIb}(e), $\alpha_2(T') = \alpha_2(T'-v)+1 \ge (|S|-2)+1 = |S|-1 = \alpha_2(T)-1$. In both cases, $\alpha_2(T') \ge \alpha_2(T)-1$. Consequently, $\alpha_2(T) = \alpha_2(T')+1$. Thus, $\alpha_2(T)-\gamma_2(T) = \alpha_2(T')-\gamma_2(T') = c$.

\textit{Case 3. $T$ is obtained from $T'$ by Operation $\mathcal{O}_3$.}
In this case, the attacher $v$ is an arbitrary vertex of $T'$. Let $u_1u_2u_3$ be the path added to $T'$ and $vu_2$ the edge added to $T'$ to obtain $T$. Since $u_1$ and $u_3$ are leaves of $T$, we note that $\{u_1,u_3\} \subset D$. If $u_2 \in D$, then we can simply replace $u_2$ in $D$ by the vertex $v$. Hence, we may assume that $u_2 \notin D$. The set $D \setminus \{u_1,u_3\}$ is therefore a $2$-dominating set of $T'$, and so $\gamma_2(T') \le \gamma_2(T)-2$. Every $2$-dominating set of $T'$ can be extended to a $2$-dominating set of $T$ by adding to it the leaves $u_1$ and $u_3$, implying that $\gamma_2(T) \le \gamma_2(T')+2$. Consequently, $\gamma_2(T) = \gamma_2(T')+2$. Every $2$-independent set of $T'$ can be extended to a $2$-independent set of $T$ by adding to it the leaves $u_1$ and $u_3$, implying that $\alpha_2(T) \ge \alpha_2(T')+2$. Suppose that $u_2 \in S$. Then, at most one of $u_1$ and $u_3$ belong to $S$. Renaming $u_1$ and $u_3$ if necessary, we may assume that $u_1 \notin S$. In this case, we can simply replace $u_2$ in $S$ with $u_1$. Hence, we may assume that $u_2 \notin S$, and so $\{u_1,u_3\} \subset S$. The set $S \setminus \{u_1,u_3\}$ is therefore a $2$-independent set of $T'$, and so $\alpha_2(T') \ge \alpha_2(T)-2$. Consequently, $\alpha_2(T) = \alpha_2(T')+2$. Thus, $\alpha_2(T)-\gamma_2(T) = \alpha_2(T')-\gamma_2(T') = c$.

\textit{Case 4. $T$ is obtained from $T'$ by Operation $\mathcal{O}_4$.}
In this case, $T_{\pdi} \in \{T_1,T_2,T_3,T_5,T_6,T_7$, $T_9,T_{10}, T_{14}\}$. Let $P \colon u_1u_2u_3$ be the path added to $T'$ and $vu_1$ the edge added to $T'$ to obtain $T$. Every $2$-dominating set of $T'$ can be extended to a $2$-dominating set of $T$ by adding to it vertices $u_2$ and $u_3$, implying that $\gamma_2(T) \le \gamma_2(T')+2$. Since $u_3$ is a leaf of $T$, we have $u_3 \in D$. If $u_2 \in D$, then we can simply replace $u_2$ in $D$ by $u_1$. If $u_2 \notin D$, then $u_1 \in D$ in order to dominate the vertex $u_2$ twice. Thus, we may assume that $D \cap \{u_1,u_2,u_3\} = \{u_1,u_3\}$. Suppose $T_{\pdi} \in \{T_1,T_2,T_5,T_6,T_9,T_{10},T_{14}\}$. By Observation \ref{ob:PDIb}(a), \ref{ob:PDIb}(b), and \ref{ob:PDIb}(d), the set $D$ can be chosen so that $D \cap V(T_{\pdi})$ are the square vertices of $T_{\pdi}$ illustrated in Figure 3, noting that in this case the vertex $v = v(T_{\pdi})$ is the white vertex of $T_{\pdi}$. This implies that either $v \in D$ or $v \notin D$ and $v$ is dominated twice by vertices of $D \setminus \{u_1,u_3\}$. In both cases, the set $D \setminus \{u_1,u_3\}$ is a $2$-dominating set of $T'$. Therefore, $\gamma_2(T') \le |D|-2 = \gamma_2(T)-2$. Consequently, $\gamma_2(T) = \gamma_2(T')+2$. Suppose $T_{\pdi} \in\{T_3,T_7\}$. If $v \notin D$, then since the two neighbors of $v$ in $T_{\pdi}$ have degree at most $2$, the set $D$ must contain both neighbors of $v$ in $T_{\pdi}$, for otherwise a neighbor of $v$ not in $D$ would not be dominated twice by vertices of $D$, a contradiction. Therefore, either $v \in D$ or $v\notin D$ and both neighbors of $v$ belong to $D$. In both cases, the set $D \setminus \{u_1,u_3\}$ is a $2$-dominating set of $T'$. Therefore, $\gamma_2(T') \le \gamma_2(T)-2$. Consequently, $\gamma_2(T) = \gamma_2(T')+2$. If $u_1 \in S$, then at most one of $u_2$ and $u_3$ belong to $S$, and in this case we can simply replace the two vertices of $P$ that belong to $S$ with the vertices $u_2$ and $u_3$. If $u_1 \notin S$, then $\{u_2,u_3\} \subset S$. Hence, we may assume that $S \cap \{u_1,u_2,u_3\} = \{u_2,u_3\}$. The set $S \setminus \{u_2,u_3\}$ is therefore a $2$-independent set of $T'$, and so $\alpha_2(T') \ge \alpha_2(T)-2$. Every $2$-independent set of $T'$ can be extended to a $2$-independent set of $T$ by adding to it the vertices $u_2$ and $u_3$, implying that $\alpha_2(T) \ge \alpha_2(T')+2$. Consequently, $\alpha_2(T) = \alpha_2(T')+2$. Thus, $\alpha_2(T)-\gamma_2(T) = \alpha_2(T')-\gamma_2(T') = c$.

\textit{Case 5. $T$ is obtained from $T'$ by Operation $\mathcal{O}_5$.}
 In this case, $T_{\pdi} \cong T_6$. Let $u_1u_3$ be the path added to $T'$ and $u_2$ the new vertex added to $T'$, and let $v_1u_1$ and $v_2u_2$ be the two edges added to $T'$ to obtain $T$. Since $u_2$ and $u_3$ are leaves of $T$, we note that $\{u_2,u_3\} \subset D$. If $u_1 \in D$, then we can simply replace $u_1$ in $D$ by $v_1$. If $u_1 \notin D$, then $v_1 \in D$ in order to dominate $u_1$ twice. Therefore, we can choose $D$ so that $D \cap \{u_1,u_2,u_3,v_1\} = \{u_2,u_3,v_1\}$. Let $v_1w_1w_2w_3v_2$ be the $(v_1,v_2)$-path. If $w_1 \in D$, we can replace $w_1$ in $D$ by $w_2$. If $w_1 \notin D$, then $w_2 \in D$ in order to dominate $w_1$ twice. Hence, we can choose $D$ so that $D \cap \{w_1,w_2\} = \{w_2\}$. If $w_3 \in D$, we can replace $w_3$ in $D$ by $v_2$. If $w_3 \notin D$, then $v_2 \in D$ in order to dominate $w_3$ twice. Hence, we can choose $D$ so that $D \cap \{w_3,v_2\} = \{v_2\}$. The set $D \setminus \{u_2,u_3\}$ is therefore a $2$-dominating set of $T'$, implying that $\gamma_2(T') \le |D| - 2 = \gamma_2(T)-2$. By Observation \ref{ob:PDIb}(b), the set $D'$ contains the vertex $v_1$, and can therefore be extended to a $2$-dominating set of $T$ by adding to it the leaves $u_2$ and $u_3$, implying that $\gamma_2(T) \le |D'|+2 = \gamma_2(T')+2$. Consequently, $\gamma_2(T) = \gamma_2(T')+2$. If $v_1 \in S$, then at most one of $u_1$ and $u_3$ belong to $S$, and in this case we can replace the vertices in the set $\{u_1,u_3,v_1\}$ that belong to $S$ with the vertices $u_1$ and $u_3$. If $v_1 \notin S$, then $\{u_1,u_3\} \subset S$. Hence, we may assume that $S \cap \{u_1,u_3,v\} = \{u_1,u_3\}$. If $w_1 \notin S$, then $w_3 \in S$ and we can replace $w_3$ in the set $S$ with $w_1$. Hence, we may assume that $w_1 \in S$. If $w_2 \notin S$, then $w_3 \in S$ and we can replace $w_3$ in the set $S$ with $w_2$. Hence, we may assume that $w_2 \in S$ and $w_3 \notin S$. If $u_2 \notin S$, then $v_2 \in S$ and we can replace $v_2$ in the set $S$ with $u_2$. Hence, we may assume that $u_2 \in S$. With these assumptions, we note that both cases $v_2 \in S$ and $v_2 \notin S$ may possibly occur. However in both cases, the set $(S \setminus \{u_1,u_2,u_3,w_2\}) \cup \{v_1,w_3\}$ is a $2$-independent set of $T'$, and so $\alpha_2(T') \ge |S|-2 = \alpha_2(T)-2$. By Observation \ref{ob:PDIb}(b), we note that $S' \cap \{v_1,w_1,w_2,w_3\} = \{v_1,w_1,w_3\}$. We note that both cases $v_2 \in S'$ and $v_2 \notin S'$ may possibly occur. However in both cases, the set $(S' \setminus \{v_1,w_3\}) \cup \{u_1,u_2,u_3,w_2\}$ is a $2$-independent set of $T$, implying that $\alpha_2(T) \ge |S'|+2 = \alpha_2(T')+2$. Consequently, $\alpha_2(T) = \alpha_2(T')+2$. Thus, $\alpha_2(T)-\gamma_2(T) = \alpha_2(T')-\gamma_2(T') = c$.

\textit{Case 6. $T$ is obtained from $T'$ by Operation $\mathcal{O}_6$.}
In this case, $T_{\pdi} = T_{14}$. Let $u_1u_2u_3$ be the path added to $T'$, and let $v_1u_1$ and $v_2u_2$ be the two edges added to $T'$ to obtain $T$. Since $u_3$ is a leaf of $T$, we note that $u_3 \in D$. Further, since $(V(T_{\pdi})\setminus {v_1}) \cup \{u_2\}$ induces a PDI-subgraph $T_{14}$, $u_2 \in D$ by Observation \ref{ob:PDIb}. In order to dominate some vertex $u_1$, $u_1 \in D$, $v_1 \in D$ twice. We can replace $u_1$ in the set $D$ by $v_1$ if $u_1 \in D$. Hence, we may assume that $v_1\in D$. But now, $v_1$ dominates $v_2$ in $T'$. Therefore, $D\setminus\{u_2,u_3\}$ is a $2$-dominating set of $T'$ and $\gamma_2(T') \le |D|-2 = \gamma_2(T)-2$. By Observation \ref{ob:PDIb}(b), $v_1 \in D'$. Now, $D' \cup \{u_2,u_3\}$ is a $2$-dominating set of $T$, implying $\gamma_2(T) \le \gamma_2(T')+2$. Consequently, $\gamma_2(T) = \gamma_2(T')+2$. We note that by definition, $|S \cap \{u_1,u_2,u_3\}| \le 2$. Hence, $S \setminus \{u_1,u_2,u_3\}$ is a $2$-independent set of $T'$, and further $\alpha_2(T') \ge \alpha_2(T)-2$. By Observation \ref{ob:PDIb}(b), we note that $v_2\notin S'$. Therefore, $S' \cup \{u_2,u_3\}$ is a $2$-independent set of $T$, implying $\alpha_2(T) \ge |S'|+2 = \alpha_2(T')+2$. Consequently, $\alpha_2(T) = \alpha_2(T')+2$. Thus, $\alpha_2(T)-\gamma_2(T) = \alpha_2(T')-\gamma_2(T') = c$.
\end{proof}

We are now in position to prove the following lemma.

\begin{lem}\label{lem2}
Every tree of the family $\cT$ is a $(\gamma_2,\alpha_2)$-tree.
\end{lem}
\begin{proof}
We proceed by induction on the order $n \ge 1$ of a tree $T \in \cT$. Every tree of order at most $4$ is a $(\gamma_2,\alpha_2)$-tree. This establishes the base. For the inductive hypothesis, let $n \ge 5$ and assume that every tree of order less than $n$ that belongs to the family $\cT$ is a $(\gamma_2,\alpha_2)$-tree. Let $T$ be a tree of order $n$ that belongs to the family $\cT$. Then there exists a sequence $T_0,T_1,\ldots,T_k$ of trees such that $T_0$ is a tree of order at most $4$, $T_k = T$, and for $i \in [k]$, the tree $T_i$ can be obtained from the tree $T_{i-1}$ by one of the Operations $\mathcal{O}_1,\mathcal{O}_2,\ldots,\mathcal{O}_6$. Let $T' = T_{k-1}$ and note that $T' \in \cT$ and $T'$ has order less than $n$. Applying the inductive hypothesis to the tree $T'$, we have that $T'$ is a $(\gamma_2,\alpha_2)$-tree. Thus, $\gamma_2(T') = \alpha_2(T')$. The tree $T$ can be obtained from $T'$ by applying one of the Operations $\mathcal{O}_1$, $\mathcal{O}_2$, $\mathcal{O}_3$, $\mathcal{O}_4$, $\mathcal{O}_5$, or $\mathcal{O}_6$. Therefore, by Lemma \ref{lemma1}, $\alpha_2(T)-\gamma_2(T) = \alpha_2(T')-\gamma_2(T') = 0$; or, equivalently, $\gamma_2(T) = \alpha_2(T)$. Therefore, $T$ is a $(\gamma_2,\alpha_2)$-tree.
\end{proof}

\begin{lem}\label{lem3}
Every $(\gamma_2,\alpha_2)$-tree belongs to the family $\cT$.
\end{lem}
\begin{proof}
We show that if $T$ is a $(\gamma_2,\alpha_2)$-tree, then $T \in \cT$. We proceed by induction on the order $n \ge 1$ of a $(\gamma_2,\alpha_2)$-tree $T$. If $n \le 4$, then $T \in \cT$. This establishes the base case. For the inductive hypothesis, let $n \ge 5$ and assume that every $(\gamma_2,\alpha_2)$-tree of order less than $n$ belongs to the family $\cT$. Let $T$ be a $(\gamma_2,\alpha_2)$-tree of order $n$. We show that $T \in \cT$. If $T$ is a star, then this is immediate since $T$ can be obtained from a path $P_3$ by repeated applications of Operation $\mathcal{O}_1$. Hence, we may assume that $\diam(T) \ge 3$.

We will frequently use the following three facts throughout the remaining proof.

\begin{fact}\label{fact1}
If $T$ contains a set $U$ of vertices such that $T$ can be obtained from the tree $T-U$ by applying Operation $\mathcal{O}_i$ for some $i \in [5]$, then $T \in \cT$.
\end{fact}
\begin{claimproof} Let $U$ be a set of vertices of $T$, and let $T' = T-U$. If $T$ can be obtained from the tree $T'$ by applying Operation $\mathcal{O}_i$ for some $i \in [5]$, then, by Lemma \ref{lemma1}, $\alpha_2(T)-\gamma_2(T) = \alpha_2(T')-\gamma_2(T')$. By supposition, $T$ is a $(\gamma_2,\alpha_2)$-tree, and so $\alpha_2(T)-\gamma_2(T) = 0$. Therefore, $\alpha_2(T')-\gamma_2(T') = 0$, and so $T'$ is a $(\gamma_2,\alpha_2)$-tree. Applying the inductive hypothesis to $T'$, the tree $T' \in \cT$. Since $T$ can be restored by applying Operation $\mathcal{O}_i$ to the tree $T' \in \cT$, the tree $T \in \cT$.
\end{claimproof}

\medskip

As a consequence of Fact \ref{fact1}, we have the following result.

\begin{fact}\label{fact2}
If $T$ contains a PDI-subtree $T_{\pdi}^{\mathcal{O}_i}$ for some $i \in \{1,2,4,5,6\}$, then $T \in \cT$.
\end{fact}

\begin{claimproof} Clearly, the result is true for $i\in\{1,2,4,5\}$ by Fact \ref{fact1}. Hence, we may assume that $T$ contains a PDI-subtree $T_{14}^{\mathcal{O}_6}$. Therefore, by definition, there is a tree $T'$ such that $T$ is obtained by applying Operation $\mathcal{O}_6$ to $T'$. Moreover, by Lemma \ref{lemma1}, $\alpha_2(T')-\gamma_2(T')=\alpha_2(T)-\gamma_2(T)$, and so we conclude that $T'$ is a $(\gamma_2,\alpha_2)$-tree. By induction, $T'\in\cT$, implying $T\in \cT$.
\end{claimproof}

\medskip

\begin{fact}\label{fact3}
Let $U$ be a set of vertices in $T$, and let $T' = T[U]$. Let $D'$ be a $2$-dominating set in $T'$ and let $S'$ be a $2$-independent set in $T'$. If $|D'| < |S'|$ and $\partial(U) \cap S' = \emptyset$, then we have a contradiction to the choice of $T$.
\end{fact}
\begin{claimproof}
Suppose that $|D'| < |S'|$ and $\partial(U) \cap S' = \emptyset$. Every $2$-dominating set of $T-U$ can be extended to a $2$-dominating set of $T$ by adding to it the set $D'$, implying that $\gamma_2(T) \le \gamma_2(T-U)+|D'|$. Since no vertex in $S'$ is adjacent to a vertex in $V(T) \setminus U$, every $2$-independent set in $T-U$ can be extended to a $2$-independent set in $T$ by adding to it the set $S'$. Thus, $\alpha_2(T) \ge \alpha_2(T-U)+|S'|$. By Theorem \ref{t:relate}, $\gamma_2(T-U) \le \alpha_2(T-U)$. Thus, $\alpha_2(T) = \gamma_2(T) \le \gamma_2(T-U)+|D'| \le \alpha_2(T-U)+|D'| < \alpha_2(T-U)+|S'| \le \alpha_2(T)$, a contradiction.
\end{claimproof}

\medskip

We proceed further with the following series of claims.

\begin{claim}\label{c:lem3.1}
If $T$ has a support vertex adjacent to at least three leaves, then $T \in \cT$.
\end{claim}
\begin{claimproof} Let $v$ be a support vertex adjacent to at least three leaves. Let $u$ be an arbitrary leaf adjacent to $v$ and let $U = \{u\}$. Since $T$ can be obtained from the tree $T-U$ by applying Operation $\mathcal{O}_1$ with the vertex $v$ as the attacher in $T-U$, Fact \ref{fact1} implies that $T \in \cT$. \end{claimproof}

\medskip

By Claim \ref{c:lem3.1}, we may assume that every support vertex of $T$ is adjacent to at most two leaves, for otherwise the desired result follows. By our earlier assumptions, $\diam(T) \ge 3$. Let $P$ be a longest path in $T$ and suppose that $P$ is an $(r_1,r_2)$-path. Necessarily, $r_1$ and $r_2$ are leaves in $T$. Renaming $r_1$ and $r_2$ if necessary, we may assume that the degree of the support vertex adjacent to $r_1$ is at most the degree of the support vertex adjacent to $r_2$. We now let $r = r_1$ and root the tree $T$ at the vertex $r$.

We call a vertex of degree at least $2$ in $T$ a {\it large vertex}. Let $\mathcal{L}$ be the set of large vertices in $T$. For each vertex $w \in \mathcal{L}$, let $\ell(w)$ be a leaf at maximum distance from $w$ in $T$ that belongs to the maximal subtree, $T_w$, at $w$. In particular, we note that $w$ belongs to the $(r,\ell(w))$-path.

\begin{claim}\label{c:lem3.2}
If $w \in \mathcal{L}$ and $T_w$ is a PDI-subtree $T_1$ in $T$, then $T \in \cT$.
\end{claim}
\begin{claimproof} Suppose that $w \in \mathcal{L}$ and $D[w]$ induces the PDI-subtree $T_1$. Let $U = D[w]$. In this case, $U$ consists of $w$ and its two children. Since $T$ can be obtained from the tree $T-U$ by applying Operation $\mathcal{O}_3$ with the parent of $w$ in $T$ as the attacher in $T-U$, Fact \ref{fact1} implies that $T \in \cT$.
\end{claimproof}

\medskip

By Claim \ref{c:lem3.2}, we may assume that if $w \in \mathcal{L}$, then $T_w$ is not a PDI-subtree $T_1$ in $T$, for otherwise the desired result follows. We define $\mathcal{B}_0 = \{B_1\}$, $\mathcal{B}_1 = \{B_2\}$, $\mathcal{B}_2 = \{B_3\}$, $\mathcal{B}_3 = \{B_4,B_5,B_6\}$, $\mathcal{B}_4 = \{B_7,B_8\}$, $\mathcal{B}_5 = \{B_9,B_{10}\}$, and $\mathcal{B}_i = \emptyset$ for $i \ge 6$. If $T_{\pdi}$ is a PDI-subtree of $T$ and $T_{\pdi}$ is isomorphic to a tree in the family $\mathcal{B}_i$ for some $i \ge 0$, then we say that $T_{\pdi}$ is a {\it PDI-subtree of $\mathcal{B}_i$ in $T$}.

\begin{claim}\label{c:lem3.3}
\label{claim33}
For every vertex $w \in \mathcal{L}$, the subtree $T_w$ is a PDI-subtree of $\mathcal{B}_{d(w,\ell(w))}$ in $T$ or $T \in \cT$.
\end{claim}
\begin{claimproof}
We proceed by induction on the distance, $d(w,\ell(w))$, from $w$ to the leaf $\ell(w)$. Suppose that $d(w,\ell(w)) = 1$. In this case, every child of $w$ is a leaf. Since every support vertex in $T$ is adjacent to at most two leaves, $T_w$ is a PDI-subtree $B_2$ or $T_1$. By Claim \ref{c:lem3.2}, we have $T_w \cong B_2$, implying that $T_w$ is a PDI-subtree of $\mathcal{B}_1$ in $T$, where $w$ is the white vertex of $B_2$.

\textit{Case 1. Assume that $d(w,\ell(w)) = 2$.}
For each child $z$ of $w$, we note by induction that $T_z$ is a PDI-subtree of $\mathcal{B}_0\cup\mathcal{B}_1$ in $T$. Further, since $d(w,\ell(w)) = 2$, at least one child of $w$, say $x_1$, satisfies $T_{x_1} \cong B_2$. Let $y_1$ be the leaf adjacent to $x_1$. If $d_T(w) = 2$, then $T_w \cong B_3$, implying that $T_w$ is a PDI-subtree of $\mathcal{B}_3$ in $T$, where $w$ is the white vertex of $B_3$. Hence, we may assume that $d_T(w) \ge 3$, for otherwise the desired result follows. Let $x_2$ be a child of $w$ different from $x_1$.

Now suppose that $T_{x_2} \cong B_2$. Let $y_2$ be the leaf adjacent to $x_2$, and let $U = \{w,x_1,x_2$, $y_1, y_2\}$. Let $T' = T[U]$. Then, $D' = \{w,y_1,y_2\}$ is a $2$-dominating set in $T'$ and $S' = \{x_1,x_2,y_1,y_2\}$ is a $2$-independent set in $T'$. Since $|D'| < |S'|$ and $\partial(U) \cap S' = \emptyset$, we have a contradiction by Fact \ref{fact3}. Hence, $T_{x_2} \cong B_1$. We now let $U = \{x_2\}$ and note that $T[\{w,x_1,y_1\}] \cong T_2$ is a PDI-subtree of $T-U$. Thus, $T$ contains a PDI-subtree $T_2^{\mathcal{O}_1}$ with the vertex $w$ as the attacher in this subtree, and so, by Fact \ref{fact2}, $T \in \cT$.

\textit{Case 2. Assume that $d(w,\ell(w)) = 3$.}
For each child $z$ of $w$, we note by induction that $T_z$ is a PDI-subgraph of $\mathcal{B}_0\cup\mathcal{B}_1\cup\mathcal{B}_2$ in $T$. Further, since $d(w,\ell(w)) = 3$, at least one child of $w$, say $x_1$, satisfies $T_{x_1} \cong B_3$. If $d_T(w) = 2$, then $T_w$ is a PDI-subtree $B_4 \in \mathcal{B}_3$ in $T$ where $w$ is the white vertex of $B_4$. Hence, we may assume that $d_T(w) \ge 3$, for otherwise the desired result follows. If $w$ has a child $x_2$ such that $T_{x_2}$ is a PDI-subtree $B_2$ in $T$ where $x_2$ is the white vertex of $B_2$, then let $U = D[x_1]$. Now $T[D(x_2) \cup \{w\}] \cong T_2$ is a PDI-subtree of $T-U$. Therefore, $T$ contains a PDI-subtree $T_2^{\mathcal{O}_3}$ with the vertex $w$ as the attacher in this subtree, and so, by Fact \ref{fact2}, $T \in \cT$. Thus for each child $x$ of $w$ we have that $T_z$ is a PDI-subtree $B_1$ or $B_3$.

Now assume that $w$ has at least three children, say $x_1$, $x_2$ and $x_3$. If $T_{x_2}$ and $T_{x_3}$ are two PDI-subtrees $B_1$, then let $U = D[x_1]$. Now $T[D(x_2) \cup D(x_3) \cup \{w\}] \cong T_1$ is a PDI-subtree of $T-U$. Thus, $T$ contains a PDI-subtree $T_1^{\mathcal{O}_4}$ with the vertex $w$ as the attacher in this subtree, and so, by Fact \ref{fact2}, $T \in \cT$. If $T_{x_2}$ and $T_{x_3}$ are two PDI-subtrees such that $T_{x_2} \cong B_1$ and $T_{x_3} \cong B_3$, or $T_{x_2},T_{x_3} \cong B_3$, then, by defining $U = D[x_1]$, $T[D[x_2] \cup D[x_3] \cup \{w\}]$ is a PDI-subtree $T_5$ or $T_{10}$, respectively. Now $T$ contains a PDI-subtree $T_5^{\mathcal{O}_4}$ or $T_{10}^{\mathcal{O}_4}$ with the vertex $w$ as the attacher in this subtree, and so, by Fact \ref{fact2}, $T \in \cT$. Therefore, $w$ has degree $3$ in $T$ and there are at most two children $x_1$ and $x_2$ such that $T_{x_1}$ and $T_{x_2}$ are PDI-subtrees satisfying $T_{x_1} \cong B_3$, and $T_{x_2} \cong B_1$ or $T_{x_2} \cong B_3$. Hence, $T_w \cong B_5 \in \mathcal{B}_3$ or $T_w \cong B_6 \in \mathcal{B}_3$, where $w$ is the white vertex in both cases. Hence, $T_w$ is a PDI-subtree of $\mathcal{B}_3$ in $T$.

\textit{Case 3. Assume that $d(w,\ell(w)) = 4$. }
For each child $z$ of $w$, we note by induction that $T_z$ is a PDI-subtree of $\mathcal{B}_0\cup\mathcal{B}_1\cup\mathcal{B}_2 \cup \mathcal{B}_3$ in $T$. Further, since $d(w,\ell(w)) = 4$, at least one child of $w$, say $x_1$, satisfy that $T_{x_1}$ is isomorphic to $B_4,B_5$, or $B_6$.

Assume that $d_T(w) = 2$. If $T_{x_1} \cong B_4$, then $T_w \cong B_7$, implying that $T_w$ is a PDI-subtree of $\mathcal{B}_4$ in $T$, where $w$ is the white vertex of $B_7$. Therefore, $T_w \in \{B_5,B_6\}$. Further, since $w \in \mathcal{L}$, we may assume that $w$ is the child of $y$, and that $z$ is a child of $x_1$ of degree $2$. We note that $z$ has distance $2$ to a leaf in $T_{x_1}$. Now let $U = D[z]$. Depending on whether $T_{x_1} \cong B_5$ or $T_{x_1} \cong B_6$, $T_y-U$ is a PDI-subtree $T_3$ or $T_7$, respectively. Thus, $T$ contains a PDI-subtree $T_3^{\mathcal{O}_4}$ or $T_7^{\mathcal{O}_4}$ with the vertex $x_1$ as the attacher in this subtree, and so, by Fact \ref{fact2}, $T \in \cT$. Hence, we may assume that $d_T(w) \ge 3$, for otherwise the desired result follows. Therefore, let $x_2$ be a child of $w$ different from $x_1$.

Recall, for any child $z$ of $w$, $T_z$ is a PDI-subtree of $\mathcal{B}_0 \cup \mathcal{B}_1 \cup \mathcal{B}_2 \cup \mathcal{B}_3$ in $T$, and $T_{x_1}$ is a PDI-subtree of $\{B_4,B_5,B_6\}$.

If there is a child, renaming vertices if necessary, say $x_2$, of $w$ such that $T_{x_2}$ is a PDI-subtree of \linebreak$\{B_2,B_4,B_5,B_6\}$ in $T$, then let $U = \{w\} \cup D[x_1] \cup D[x_2]$. By a simple case analysis and Observation \ref{ob:PDIb}, one can readily observe that there is a $2$-dominating set $D'$ and a $2$-independent set $S'$ in $T[U]$ such that $|D'| < |S'|$ and $\partial(U) \cap S' = \emptyset$. Therefore, we have a contradiction by Fact \ref{fact3}, implying that there is a child, say $x_2$, distinct from $x_1$ such that $T_{x_2}$ is a PDI-subtree of $\{B_1,B_3\}$ in $T$.

Assume that $T_{x_1} \cong B_4$. If $T_{x_2} \cong B_3$, then let $U = D[x_2]$. Thus, $T[D[x_1] \cup \{w\}]$ is a PDI-subtree $T_6$, implying that $T$ contains a PDI-subtree $T_6^{\mathcal{O}_4}$ with the vertex $w$ as the attacher in this subtree, and so, by Fact \ref{fact2}, $T \in \cT$. Therefore, any child of $w$, distinct from $x_1$, is a leaf. If $w$ has at least three children, say $x_1$, $x_2$, and $x_3$, then let $U = \{x_3\}$. Now $T[D[x_1] \cup \{x_2,w\}]$ is a PDI-subtree $T_8$. Hence, $T$ contains a PDI-subtree $T_8^{\mathcal{O}_1}$ with the vertex $w$ as the attacher in this subtree, and so, by Fact \ref{fact2}, $T \in \cT$. Therefore, $w$ has two children, namely $x_1$ and a leaf $x_2$. Further, let $U$ be the set of those two vertices in $D(x_1)$ which have largest and second largest distance to $x_1$ in $T$. Now $T[(D[x_1]\setminus U) \cup \{x_2,w\}]$ is a PDI-subtree $T_4$ implying that $T$ contains a PDI-subtree $T_4^{\mathcal{O}_2}$, and so, by Fact \ref{fact2}, $T \in \cT$.

Assume that $T_{x_1} \cong B_5$. If $T_{x_2} \cong B_3$, then let $U = D(x_2)$ implying that $T[D[x_1] \cup \{x_2,w\}]$ is a PDI-subtree $T_{13}$ in $T$. Therefore, $T$ contains a PDI-subtree $T_{13}^{\mathcal{O}_2}$ and so, by Fact \ref{fact2}, $T \in \cT$. Hence, any child of $w$, distinct from $x_1$, is a leaf. If $w$ has at least three children, say $x_1$, $x_2$, and $x_3$, then let $U = V(T_w)$. By Observation \ref{ob:PDIb}, one can readily observe that there is a $2$-dominating set $D'$ and a $2$-independent set $S'$ in $T[U]$ such that $|D'| < |S'|$ and $\partial(U) \cap S' = \emptyset$, implying a contradiction by Fact \ref{fact3}. Therefore, $w$ has only two children, say $x_1$ and $x_2$, for otherwise the desired result follows. Now $T_w\cong B_8$ and so $T_w$ is a PDI-subtree of $\mathcal{B}_4$ in $T$, where $w$ is the white vertex of $B_8$.

Assume that $T_{x_1} \cong B_6$. If $T_{x_2} \cong B_1$, then let $U$ be the set of vertices consisting of $\ell(x_1)$ and its support vertex. Now $T[(D[x_1]\setminus U) \cup \{x_2,w\}]$ is a PDI-subtree $T_{12}$ in $T$. Therefore, $T$ contains a PDI-subtree $T_{12}^{\mathcal{O}_2}$ and so, by Fact \ref{fact2}, $T \in \cT$. If $T_{x_2} \cong B_3$, then let $U = D(x_1) \cup D[x_2]$. Trivially, $T[U]$ is $3P_3$. Therefore, any $\gamma_2(T[U])$-set $D'$ contains all leaves of this forest. Furthermore, any $\gamma_2(T-U)$-set $D^*$ contains $x_1$ since it is a leaf in $T-U$. Moreover, there is an $\alpha_2(T[U])$-set $S'$ such that the leaves in $T[U]$, which are joined by bridges to $x_1$ or $x_2$, are not in $S'$. Therefore, $S' \cup S^*$ is a $2$-independent set for any $\alpha_2(T-U)$-set $S^*$. Furthermore, $D' \cup (D^*\setminus\{x_1\})$ is a $2$-dominating set, since $x_1$ and $w$ are dominated by two vertices of $D'$ and one vertex of $D'$ and $D^*$, respectively. Therefore, since $\gamma_2(T-U)\leq \alpha_2(T-U)$, by Theorem \ref{t:relate}, $\gamma_2(T)\leq \gamma_2(T-U)+5 < \alpha_2(T-U)+6\leq  \alpha_2(T)$, contradicting the choice of $T$.

\textit{Case 4. Assume that $d(w,\ell(w)) = 5$.}
For each child $z$ of $w$, we note by induction that $T_z$ is a PDI-subgraph of $\mathcal{B}_0\cup\mathcal{B}_1\cup\mathcal{B}_2\cup\mathcal{B}_3\cup\mathcal{B}_4$ in $T$. Further, since $d(w,\ell(w)) = 5$, at least one child of $w$, say $x_1$, satisfies that $T_{x_1}$ is isomorphic to $B_7$ or $B_8$.

Firstly, assume that $x_1$ has degree $2$. If $T_{x_1} \cong B_7$, then $T_w\cong B_9 \in \mathcal{B}_5$, implying that $T_w$ is a PDI-subtree of $\mathcal{B}_4$ in $T$, where $w$ is the white vertex of $B_9$. Therefore, $T_{x_1} \cong B_8$. Let $U = D[x_1]$. Now there is a $2$-dominating set $D'$ and a $2$-independent set $S'$ in $T[U]- x_1$ such that $|D'| < |S'|$. Further, we note that $w$ is a leaf in $T-U$. Therefore, there is a $\gamma_2(T_U)$-set $D^*$ and an $\alpha_2(T_U)$-set $S^*$ containing $w$. Clearly, $S^*\cup S'$ is a $2$-independent set in $T$ and $D^*\cup D'$ is a $2$-dominating set for $T-x_1$ by definition. Moreover, $x_1$ is dominated by its leaf neighbor and the vertex $w$, implying that $D^*\cup D'$ is a $2$-dominating set in $T$. Since $|D^*| \le |S^*|$, by Theorem \ref{t:relate}, $\gamma_2(T) \le |D^*\cup D'|= |D^*|+|D'| <|S^*|+|S'| =|S^*\cup S'| \le \alpha_2(T)$, a contradiction to the choice of $T$. Hence, we may assume that $d_T(w) \ge 3$, for otherwise the desired result follows. Therefore, let $x_2$ be a child of $w$ different from $x_1$.

Assume that $T_{x_1} \cong B_7$. Let $U = D[x_1] \cup D[x_2] \cup \{w\}$. By some simple case analysis and Observation \ref{ob:PDIb}, one can readily observe that there is a $\gamma_2(T[U])$-set $D'$ and an $\alpha_2(T)$-set $S'$ such that $|D'| < |S'|$ and $\partial(U) \cap S' = \emptyset$ for $T_{x_2}\in \mathcal{B}_5$, a contradiction by Fact \ref{fact3}.

Assume $T_{x_1} \cong B_8$. If $T_{x_2}$ is a PDI-subtree of $\{B_2,B_4,B_5,B_6,B_7,B_8\}$ in $T$, then let $U = \{w\} \cup D[x_1] \cup D[x_2]$. Again, by some simple case analysis and Observation \ref{ob:PDIb}, one can readily observe that there is a $\gamma_2(T[U])$-set $D'$ and an $\alpha_2(T)$-set $S'$ such that $|D'| < |S'|$ and $\partial(U) \cap S' = \emptyset$, a contradiction by Fact \ref{fact3}. If $T_{x_2} \cong B_3$, then let $U = D[x_2]$. Now $T[D[x_1] \cup \{w\}]$ is a PDI-subtree $T_{14}$. Hence, $T$ contains a PDI-subtree $T_{14}^{\mathcal{O}_4}$ with the vertex $w$ as the attacher in this subtree, and so, by Fact \ref{fact2}, $T \in \cT$. It remains to consider the case that all children of $w$ different from $x_1$ are leaves of $T$. Recall that any support vertex is adjacent to at most two leaves. Further, if $w$ is adjacent to exactly one leaf, then $T_w\cong B_{10}$, implying that $T_w$ is a PDI-subtree of $\mathcal{B}_5$ in $T$, where $w$ is the white vertex of $B_{10}$. On the other hand, if $w$ is adjacent to two leaves, then let $U$ be the set containing $\ell(w)$ and its support vertex. Now $T[D[w]-U]$ is a PDI-subtree $T_{15}$ in $T$. Therefore, $T$ contains a PDI-subtree $T_{15}^{\mathcal{O}_2}$, and so, by Fact \ref{fact2}, $T \in \cT$.

\textit{Case 5. Assume that $d(w,\ell(w)) = 6$.}
 For each child $z$ of $w$, we note by induction that $T_z$ is a PDI-subtree of $\mathcal{B}_0\cup\mathcal{B}_1 \cup \mathcal{B}_2 \cup \mathcal{B}_3 \cup \mathcal{B}_4 \cup \mathcal{B}_5$ in $T$. Further, since $d(w,\ell(w)) = 6$, at least one child of $w$, say $x_1$, satisfies that $T_{x_1}$ is isomorphic to $B_9$ or $B_{10}$.

Assume that $T_{x_1} \cong B_9$. If $w$ has degree $2$, then let $U = D[x_1]$. Now there is a $2$-dominating set $D'$ and a $2$-independent set $S'$ in $T[U-{x_1}]$ such that $|D'| \le |S'|$. Moreover, since $w$ is a leaf in $T-U$, there is a $\gamma_2(T-U)$-set $D^*$ in $T-U$ and an $\alpha_2(T)$-set $S^*$ containing $w$. Hence, $x_1$ is dominated by a vertex of $D'$ and a vertex of $D^*$. Therefore, $D' \cup D^*$ is $2$-dominating in $T$ and $S' \cup S^*$ is $2$-independent in $T$. It follows, by Theorem \ref{t:relate}, $\gamma_2(T) \le |D'\cup D^*|=|D'|+|D^*| <|S'|+|S^*|=|S' \cup S^*| \le \alpha_2(T)$, a contradiction. Hence, there is a child of $w$, say $x_2$, distinct from $x_1$. If $T_{x_2}$ is a PDI-subtree of $\{B_2,B_4,B_5,B_6,B_6,B_7,B_8,B_9\}$ in $T$, then let $U = D[x_1] \cup D[x_2] \cup \{w\}$. By Observation \ref{ob:PDIb}, one can readily observe that in all cases there is a $2$-dominating set $D'$ and a $2$-independent set $S'$ in $T[U]$ such that $|D'| < |S'|$ and $\partial(S') \cap U = \emptyset$, a contradiction by Fact \ref{fact3}. If $T_{x_2} \cong B_1$, then let $U$ be the set of three vertices consisting of $\ell(x_1)$, its support vertex and $x_2$. Then $T[(D[x_1] \cup \{w\}) \setminus U]$ is a PDI-subtree $T_6$ in $T$. Therefore, $T$ contains a PDI-subtree $T_6^{\mathcal{O}_5}$, and so, by Fact \ref{fact2}, $T \in \cT$. If $T_z \cong B_3$, then let $U = D[x_2]$. Now $T[D[x_1] \cup \{w\}]$ is a PDI-subtree $T_9$ in $T$. Again, $T \in \cT$, by Fact \ref{fact2} and since $T$ contains a PDI-subtree $T_9^{\mathcal{O}_4}$.

Assume that $T_{x_1} \cong B_{10}$. If $w$ has degree $2$, then let $y$ be the parent of $w$. Now $D[w] \cup \{y\}$ induces a PDI-subtree $T_{14}^{\mathcal{O}_6}$, and so $T \in \cT$ by Fact \ref{fact2}. If $w$ has degree at least $3$, then there is a child of $w$, say $x_2$, distinct from $x_1$. Further, $T_{x_2}$ is a PDI-subtree of $\{B_1,B_2,\ldots ,B_{10}\}$ in $T$. Let $U = D[x_1] \cup D[x_2] \cup \{w\}$. By Observation \ref{ob:PDIb}, one can readily observe that there is a $2$-dominating set $D'$ and a $2$-independent set $S'$ in $T-U$ such that $|D'| < |S'|$ and $\partial(S') \cap U = \emptyset$, a contradiction by Fact \ref{fact3}.

Note that the proof for $d(w,\ell(w)) = 6$ immediately implies that $\mathcal{B}_i = \emptyset$ for $i \ge 6$. This completes the proof of Claim \ref{c:lem3.3}
\end{claimproof}

\medskip

We now return to the proof of Lemma \ref{lem3} for the last time. Recall that $T$ is a $(\gamma_2,\alpha_2)$-tree. Further, let $w$ be the support vertex of $r$. By Claim \ref{claim33}, we deduce that $T_w$ is a PDI-subtree of $\mathcal{B}_{d(w,\ell(w))}$ in $T$ or $T \in \cT$. Since the desired result follows in the latter case, we assume that $T_w$ is a PDI-subtree of $\mathcal{B}_{d(w,\ell(w))}$. Moreover, $d(w,\ell(w)) \le 5$ since $\mathcal{B}_i = \emptyset$ for $i \ge 6$. Further, $r$ is an end-vertex of a diametrical path which implies that $\diam(T) \le 6$. Now let $T_w \cong B_i$. Then, all vertices in $V(T) \setminus V(T_w)$ are leaves adjacent to $w$. Note that there is at least one leaf, namely the vertex $r$, and, by Claim \ref{c:lem3.1}, there are at most two leaves. If $T_w \cong B_1$ or $T_w \cong B_2$, then $T$ is a star, which can be obtained by applying Operation $\mathcal{O}_1$ to a path $P_3$. If $T_w$ is a PDI-subtree of $\{B_3,B_4,B_7,B_9\}$ in $T$, then, by the choice of $r$, we have exactly one leaf outside $V(T_w)$ being adjacent to $w$. Thus, $T$ is a path. Depending on the length of the path, we have $T \in \cT$ or $\gamma_2(T) < \alpha_2(T)$ which gives the result or contradicts the choice of $T$, respectively. If $T_w$ is a PDI-subtree of $\{B_5,B_8,B_{10}\}$ in $T$, then the degree of $w$ is at least $3$, but the support vertex of $\ell(w)$ has degree $2$, a contradiction to the choice of $r_1$. If $T_w \cong B_6$, then we have a contradiction since $r$ does not lie on a diametrical path.
\end{proof}

As an immediate consequence of Lemmas \ref{lem2} and \ref{lem3}, we obtain the following characterization of trees with equal $2$-domination and $2$-independence numbers.
\medskip 

\noindent \textbf{Theorem~\ref{t:tree}.} \emph{A tree is a $(\gtd,\tind)$-tree if and only if $T \in \mathcal T$.}

\begin{proof}
Let $T$ be a tree. If $T \in \cT$, then by Lemma~\ref{lem2}, $T$ is a $(\gtd,\tind)$-tree. This establishes the sufficiency. If $T$ is a $(\gtd,\tind)$-tree, then, by Lemma~\ref{lem3}, we have $T \in \mathcal T$. This proves the necessity.
\end{proof}

\label{sec:biblio}

\end{document}